%% file: main.tex
\documentclass[twoside]{article}
\usepackage[utf8]{inputenc}
\usepackage{xcolor}
\usepackage{amsmath}
\usepackage[nottoc,notlof,notlot]{tocbibind} 
\usepackage{tabularx}

\usepackage{tcolorbox}
\tcbset{boxsep=0mm,boxrule=0pt,colframe=white,arc=0mm,left=0.5mm,right=0.5mm}

\renewcommand{\thesection}{\arabic{section}}
\makeatletter
\let\l@section\l@chapter
\makeatother
\usepackage{hyperref}
\usepackage{amssymb}
\usepackage{amsthm}
\usepackage{algorithm}
\usepackage{algorithmic}
\usepackage{graphicx}
\usepackage{float}
\usepackage{upgreek}
\usepackage{thmtools}
\usepackage{thm-restate}
\input{defs}

\usepackage[accepted]{aistats2020}
\begin{document}

\runningtitle{Continuous-time Acceleration for Riemannian Optimization}

\runningauthor{Alimisis, Orvieto, B{\'e}cigneul, Lucchi}

\twocolumn[

%\aistatstitle{Continuous-time Accelerated Riemannian Optimization}
\aistatstitle{A Continuous-time Perspective for Modeling Acceleration \\ in Riemannian Optimization}

\aistatsauthor{ Foivos Alimisis \ \ \ \  Antonio Orvieto \ \ \ \ Gary B{\'e}cigneul \ \ \ \ Aurelien Lucchi}
%\aistatsaddress{Dep. of Mathematics \\
%ETH Z\"urich \And Dep. of Computer Science \\
%ETH Z\"urich \And Dep. of Computer Science \\
%ETH Z\"urich \And Dep. of Computer Science \\ ETH Z\"urich}]
%\aistatsaddress{Dep. of Mathematics \And Dep. of Computer Science \And Dep. of Computer \And Dep. of Computer Science}]
\aistatsaddress{ ETH Z\"urich, Switzerland } ]

\begin{abstract}
We propose a novel second-order ODE as the continuous-time limit of a Riemannian accelerated gradient-based method on a manifold with curvature bounded from below. This ODE can be seen as a generalization of the  ODE derived for Euclidean spaces, and can also serve as an analysis tool. We study the convergence behavior of this ODE for different classes of functions, such as geodesically convex, strongly-convex and weakly-quasi-convex. We demonstrate how such an ODE can be discretized using a semi-implicit and Nesterov-inspired numerical integrator, that empirically yields stable algorithms which are faithful to the continuous-time analysis and exhibit accelerated convergence. %Finally, we discuss the approximation error between the discrete algorithm and its ODE approximation.
\end{abstract}

\section{Introduction}
A core problem in machine learning is finding a minimum of a function $f:H \rightarrow \R$. In the vast majority of machine learning applications, $H$ represents either a Euclidean space or a Riemannian manifold.
Among the most popular types of methods to optimize $f$ are first-order methods, such as gradient descent which simply updates a sequence of iterates $\{ x_k \}$ by stepping in the opposite direction of the gradient $\nabla f(x_k)$. In the case $H = \R^n$, gradient descent as a first-order method has been shown to achieve a suboptimal convergence rate. In a seminal paper~\cite{nesterov1983method}, Nesterov showed that one can construct an optimal -- a.k.a. accelerated -- algorithm that achieves faster rates of convergence for both convex and strongly-convex functions.
%{\color{green} (I am a bit scared of reviewers picking up on (i) it's not the same algo for convex and strongly convex {\color{blue} Also the case for Euclidean, no problem here.}(ii) it's only optimal in a weird sense, for functions that are not twice differentiable, and for a number of iterations smaller than the dimension, right?)} {\color{red} We didn't say our algorithm is optimal I think? What would you want to change in the intro?}
The convergence analysis of this algorithm relies heavily on the linear structure of the space $H$ and it is not until recently that a first adaptation to Riemannian spaces has been derived in~\cite{zhang2018towards}. The algorithm in~\cite{zhang2018towards} is shown to obtain an accelerated rate of convergence for \emph{geodesically} strongly-convex functions. These functions are of particular interest as they are non-convex in the Euclidean sense and they occur in some fundamental problems~\cite{zhang2016first, zhang2018towards}.

In this manuscript, we take a different direction from previous works that have focused on analyzing the discrete-time form of Nesterov acceleration. We instead derive a continuous-time model that generalizes the work of~\cite{su2014differential} to non-Euclidean spaces. The resulting second-order ODE is shown to exhibit an approximate equivalence to Nesterov acceleration, and can therefore be used as an analysis tool. We prove theoretically that the continuous-time process corresponding to the derived differential equation has an accelerated rate of convergence for various types of functions.
As in~\cite{su2014differential}, one can also obtain different discrete-time algorithms from such an ODE. We will here focus on a discretization scheme that we show empirically to yield an accelerated rate of convergence.
%We also discuss some formal guarantees to bound the error between the algorithm steps and the corresponding ODE. \textcolor{red}{TODO Aurelien: adjust this claim depending on we can prove the result for $\lim_{h \to 0}$?}

% Interestingly, the approach presented in~\cite{zhang2018towards} was not shown to converge in the case of geodesic convexity. This is somewhat surprising given that~\cite{nesterov1983method} has shown an accelerated rate of convergence in the Euclidean convex case. In this paper, we beat the problem from a continuous-time perspective that can provide new insights about the design of accelerated algorithms. We fail to discretize these differential equations, but we provide experimental results validating the role of the curvature of the manifold in acceleration. We finally provide some minor theoretical guarantees about the fact that discrete and continuous dynamics in optimization can be really close each other.

In summary, our main contributions are:
\vspace{-3mm}
\begin{itemize}
    \setlength\itemsep{0em}
    \item We derive a second-order differential equation that can serve as an analysis tool for a Riemannian variant of accelerated gradient descent.
    \item We analyze the convergence behavior of this ODE for three different types of functions: geodesically convex, strongly-convex and weakly-quasi-convex.
    \item As a byproduct of our convergence analysis, we establish some new technical results about the Hessian of the Riemannian distance function. These results could be of general interest.
    \item We prove that in the case of Riemannian gradient descent applied to geodesically strongly convex functions, the discrete and continuous trajectories remain close. The extension of this result to an accelerated method is however non-trivial.
    \item We provide empirical results on several problems of interest in order to confirm the validity of our theoretical analysis and discretization scheme.
\end{itemize}

\section{Related work}

\paragraph{Accelerated Gradient Descent/Flow.}
The first \emph{practical} accelerated algorithm in a vector space is due to Nesterov, back in 1983~\cite{nesterov1983method}. Since then, the community has shown a deep interest in understanding the mechanism underlying acceleration. A recent trend has been to look at acceleration from a continuous-time viewpoint. In such a framework, accelerated gradient descent is seen as the discretization of a second-order ODE. In    ~\cite{su2014differential}, Su et al. formulated a second order differential equation to capture the dynamics of the classical algorithm from Nesterov in the convex case. In ~\cite{wibisono2016variational}, Wisibono et al. study continuous accelerated dynamics introducing the concept of Bregman Lagrangian. In \cite{wilson2016lyapunov}, Wilson et al. substitute the classical estimate sequences technique by a family of Lyapunov functions in both discrete and continuous time.
In ~\cite{shi2018understanding}, Shi et al. show that differential equations are rough approximators of real learning dynamics, i.e. a given algorithm can generate many continuous models. Finally, the same authors showed in~\cite{shi2019acceleration} that symplectic integration~\cite{hairer2006geometric} has deep links to Nesterov's method.
\vspace{-3mm}
\paragraph{Riemannian optimization.}Research in the field of Riemannian optimization has recently encountered a lot of interest. A seminal book in the field is~\cite{absil2009optimization} who gives a comprehensive review of many standard optimization methods except accelerated methods. More recently, \cite{zhang2016first} proved convergence rates for Riemannian gradient descent applied to the class of geodesically convex functions. Acceleration in a Riemannian framework was discussed in~\cite{liu2017accelerated} who claimed to have designed Riemannian accelerated methods with guaranteed convergence rates but as discussed in~\cite{zhang2018towards}, their method relies on finding the exact solution to a nonlinear equation and it is not clear how difficult this problem is. Subsequently, \cite{zhang2018towards} developed the first computationally tractable accelerated algorithm on a Riemannian manifold, but their approach only has provable convergence for geodesically strongly-convex objectives. In contrast, we here address the problem of achieving acceleration for the \emph{weaker} class of weakly-quasi-convex objective functions.

\section{Background}
\vspace{-1mm}
We review some basic notions from Riemannian geometry that are required in our analysis. For a full review, we refer the reader to a classical textbook, for instance~\cite{Spivak_textbook}. 
\vspace{-3mm}
\paragraph{Manifolds.}
A differentiable manifold $M$ is a topological space that is locally Euclidean. This means that for any point $x \in M$, we can find a neighborhood that is diffeomorphic to an open subset of some Euclidean space. This Euclidean space can be proved to have the same dimension, regardless of the chosen point, called the dimension of the manifold. A Riemannian manifold $(M,g)$ is a differentiable manifold equipped with a Riemannian metric $g_x$, i.e.  an inner product for each tangent space $T_xM$ at $x \in M$. We denote the inner product of $u,v \in T_x M$ with $\langle u,v \rangle_x$ or just $\langle u,v \rangle$ when the tangent space is obvious from context. Similarly we consider the norm as the one induced by the inner product at each tangent space.
\vspace{-3mm}
\paragraph{Geodesics}
Geodesics are curves $\gamma: [0,1] \rightarrow M$ of constant speed and of (locally) minimum length. They can be thought of as the Riemannian generalization of straight lines in Euclidean spaces. Geodesics are used to construct the exponential map $ \textnormal{exp}_x: T_x M \rightarrow M$, defined by $ \textnormal{exp}_x(v)= \gamma(1)$, where $\gamma$ is the unique geodesic such that $\gamma(0)=x$ and $\dot \gamma(0)=v$. The exponential map is locally a diffeomorphism. Using the notion of geodesics, we can define an intrinsic distance $d$ between two points in the Riemannian manifold $M$, as the infimum of lengths of geodesics that connect these two points. Geodesics also provide a way to transport vectors from one tangent space to another. This operation called parallel transport is usually denoted by $\Gamma_x^y: T_x M \to T_y M$. Closely linked to geodesics is the notion of injectivity radius. Given a point $x \in M$, we define the injectivity radius at $x$ (denoted $\textnormal{inj}(x)$), the radius of the biggest ball around $x$, where the exponential map $\textnormal{exp}_x$ is a diffeomorphism. We denote the inverse of the exponential map inside this ball by $\textnormal{log}_x$.
\vspace{-6mm}
\paragraph{Vector fields and covariant derivative.}
The correct notion to capture second order changes on a Riemannian manifold is called covariant differentiation and it is induced by the fundamental property of Riemannian manifolds to be equipped with a connection. The fact that a connection can always be defined in a Riemannian manifold is the subject of the fundamental theorem of Riemannian geometry. We are interested in a specific type of connection, called the Levi-Civita connection, which induces a specific type of covariant derivative. For our purpose, it will however be sufficient to define the notion of covariant derivative using the (simpler) notion of parallel transport. First, we state the definition of a vector field on a Riemannian manifold.
\begin{definition}
Let $M$ be a Riemannian manifold. A vector field $X$ in $M$ is a smooth map $X:M \rightarrow \mathcal{T}M$, where $\mathcal{T}M$ is the tangent bundle, i.e. the collection of all tangent vectors in all tangent spaces of $M$, such that $p \circ X$ is the identity ($p$ is the projection from $\mathcal{T}M$ to $M$).
\end{definition}
\vspace{-1mm}
One can see a vector field as an infinite collection of imaginary curves, the so-called integral curves (formally they are solutions of first-order differential equations on $M$).
\begin{definition}
Given two vector fields $X,Y$ in a Riemannian manifold $M$, we define the covariant derivative of $B$ along $A$ to be
\vspace{-3mm}
\begin{equation*}
    \nabla_X Y (p) := \lim_{h \rightarrow 0} \frac{\Gamma_{\gamma(h)}^{\gamma(0)} Y(\gamma(h))-Y(p)}{h},
\end{equation*}
with $\gamma$ the unique integral curve of $A$ passing from $p$.
\end{definition}

\paragraph{Geodesic convexity.}
We remind the reader of the basic definitions needed in Riemannian optimization.
\begin{definition}
A subset $A\subseteq M$ of a Riemannian manifold $M$ is called geodesically uniquely convex, if every two points in $A$ are connected by a unique geodesic.
\end{definition}
\begin{definition}
A function $f:M \rightarrow \R$ is called geodesically convex, if $f(\gamma(t)) \leq (1-t)f(p)+tf(q)$, for $t \in [0,1]$, where $\gamma$ is any geodesic connecting $p,q \in M$. 
\end{definition}

Given a function $f:M \rightarrow \R$, the notions of  differential and (Riemannian) inner product allow us to define the Riemannian gradient of $f$ at $x \in M$, which is a tangent vector belonging to the tangent space based at $x$, $T_x M$.
\begin{definition}
The Riemannian gradient $\textnormal{\textnormal{gradf}}$ of a (real-valued) function $f:M \rightarrow \R$ at a point $x \in M$, is the tangent vector at $x$, such that $\langle \textnormal{\textnormal{gradf}}(x),u \rangle = df(x)u$~\footnote{$df$ denotes the differential of $f$, i.e. $df (x)[u] = \lim_{t \to 0} \frac{f(c(t)) - f(x)}{t},$ where $c: I \to M$ is a smooth curve such that $c(0) = x$ and $\dot c(0) = u$.}, for any $u \in T_x M$.
\end{definition}

Given the notion of Riemannian gradient and covariant derivative we can define the notion of Riemannian Hessian.
\begin{definition}
Given vector fields $A,B$ in $M$, we define the Hessian operator of $f$ to be
\begin{equation*}
    \operatorname{Hess}(f)(A,B):= \langle \nabla_A \operatorname{grad}f,B \ \rangle.
\end{equation*}
\end{definition}
Using the Riemannian inner product and the Riemannian gradient, we can formulate an equivalent definition for geodesic convexity for a smooth function $f$ defined in a geodesically uniquely convex domain $A$ (the inverse of the exponential map is well-defined).
\begin{Proposition}
Let a smooth, geodesically convex function $f:A \rightarrow \R$. Then, for any $x,y \in A$,
\begin{align*}
    f(x)-f(y) \geq \langle \textnormal{\textnormal{gradf}}(y), \log_y(x) \rangle.
\end{align*}
\end{Proposition}
As in the Euclidean case, any local minimum of a geodesically convex function is a global minimum.
\newline
In a similar manner we can define geodesic strong convexity.
\begin{definition}
A smooth function $f:A \rightarrow \R$ is called geodesically $\mu$-strongly convex, $\mu >0$, if $\forall x,y \in A$
\begin{equation*}
    f(x)-f(y) \geq \langle \textnormal{\textnormal{gradf}}(y), \textnormal{\textnormal{log}}_y(x) \rangle+\frac{\mu}{2} \| \log_y(x) \|^2.
\end{equation*}
\end{definition}
If a function $f$ is geodesically strongly convex with a non-empty set of minima, then there is only one minimum and it is global.
\newline
We now generalize the well-known notion of Euclidean weak-quasi-convexity to Riemannian manifolds. For a review of this notion the reader can check \cite{guminov2017accelerated}.
\begin{definition}
A function $f:A \rightarrow \R$ is called geodesically $\alpha$-weakly-quasi-convex with respect to $c \in M$, if
\begin{equation*}
    \alpha (f(x)-f(c)) \leq -\langle \textnormal{\textnormal{gradf}}(x), \log_x(c) \rangle 
\end{equation*}
for some fixed $\alpha \in (0,1]$ and any $x \in M$.
\end{definition}
It is easy to see that weak-quasi-convexity implies that any local minimum of $f$ is also a global minimum.
\newline
Using the notion of parallel transport we can define when $f$ is geodesically L-smooth, i.e. has Lipschitz continuous gradient in a suitable differential-geometric way.
\begin{definition}
A function $f:M \rightarrow \R$ is called L-smooth if $\forall x,y \in M$ and geodesic $\gamma$ connecting them
\begin{equation*}
    \| \textnormal{\textnormal{gradf}}(x) -\Gamma_y^x {\textnormal{\textnormal{gradf}}(y)} \| \leq L l(\gamma),
\end{equation*}
where $\Gamma$ is the parallel transport along $\gamma$ and $l(\gamma)$ the length of $\gamma$.
\end{definition}
Geodesic $L$-smoothness has similar properties to its Euclidean analogue. Namely, a two times differentiable function is $L$-smooth, if and only if the norm of its Riemannian Hessian is bounded by $L$. 
\vspace{-4mm}
% Also if a function $f$ is $L$-smooth and is defined in a geodesically uniquely convex domain $A$, we have that
% \begin{equation*}
%     f(y) \leq f(x)+\langle \textnormal{\textnormal{gradf}}(x),\textnormal{\textnormal{exp}}_x^{-1}(y) \rangle+\frac{L}{2} \| \textnormal{\textnormal{exp}}_x^{-1}(y) \|^2
% \end{equation*}
% for any $x,y \in M$.
\paragraph{Curvature.}
In this paper, we make the standard assumption that the input space is not "infinitely curved". In order to make this statement rigorous, we need the notion of sectional curvature $K$, which is a measure of how sharply the manifold is curved (or how "far" from being flat our manifold is), "two-dimensionally".

%%%%%%%%%%%%%%%%%%%%%%%%%%%%%%%%%%%%%%%%%%%%%
%%%%%%%%%%%%%%%%%%%%%%%%%%%%%%%%%%%%%%%%%%%%%

\section{Hessian of the distance function}
\label{sec:Hessian_distance_function}

Before discussing the design and analysis of accelerated flows on manifolds, it is necessary to derive a crucial geometric result. During a first read, the reader may skip this section or return to it later to understand some of the technicalities in Section~\ref{sec:accelerated_ode}.

In Euclidean spaces, the law of cosines relates the lengths of the sides of a triangle to the cosine of one of its angles. One can also adapt this result to non-linear spaces as we will demonstrate next. We first derive a lemma that provides a bound on the Hessian of a variant of the the Riemannian squared distance function $-\frac{1}{2}d(X,p)^2$ for the curve $X:I \rightarrow M$ and $p \in M$. Alternatively, the Hessian of $-\frac{1}{2}d(X,p)^2$ can be seen as the covariant derivative of $\textnormal{log}_{X(t)}(p)$.
%, which is a vector field along the curve $X:I \rightarrow M$ and $p \in M$.

\begin{restatable}{lemma}{covariant}
\label{le:result for covariant derivative}
For a Riemannian manifold $M$ with curvature bounded above by $K_{\max}$ and below by $K_{\min}$ and $\textnormal{diam}(M) \leq D<$ $\begin{cases} \frac{\pi}{\sqrt{K_{\max}}} &, K_{\max}>0 \\
\infty &, K_{\max} \leq 0
\end{cases}$, we have that
\begin{equation*}
    \delta \| \dot X \| ^2 \leq\langle \nabla_ {\dot X} \textnormal{\textnormal{log}}_X(p), -\dot X \rangle \leq \zeta \| \dot X \| ^2,
\end{equation*}
where 
\vspace{-2mm}
{\small
\begin{equation*}
    \delta:= 
    \begin{cases}
    1 &, K_{\max} \leq 0\\
    \sqrt{K_{\max}} d(X,p) \cot(\sqrt{K_{\max}} d(X,p)) &, K_{\max} > 0 
    \end{cases}
\end{equation*}
and
\vspace{-2mm}
\begin{equation*}
    \zeta:= 
    \begin{cases}
    \sqrt{-K_{\min}} d(X,p) \coth(\sqrt{-K_{\min}} d(X,p)) &, K_{\min}<0 \\
    1 &, K_{\min}\geq 0
    \end{cases}.
\end{equation*}
}
\end{restatable}
\begin{restatable}{corollary}{cosinelaw}
\label{cor:cosinelaw}
Let a geodesic triangle $\Updelta abc$ in a Riemannian manifold $M$ of curvature bounded above by $K_{\max}$ and $\textnormal{diam}(M) \leq D$. We denote be $B$ the angle between the edges $ab$ and $bc$. If $K_{\max}>0$, we assume in addition that $D < \frac{\pi}{\sqrt{K_{\max}}}$. Then
\begin{align*}
     (ac)^2 \geq \delta (bc)^2 + (ab)^2-2(ab)(bc)\cos(B)
\end{align*}
where $\delta$ is defined as
\begin{equation*}
    \delta= 
    \begin{cases}
    1 &, K_{\max} \leq 0\\
    \sqrt{K_{\max}} d(q,a) \cot(\sqrt{K_{\max}} d(q,a)) &, K_{\max} > 0
    \end{cases}
\end{equation*}
for some $q \in M$ along the edge $bc$.
\end{restatable}

Note that one can also recover Lemma 5 in \cite{zhang2016first} as a corollary of Lemma \ref{le:result for covariant derivative}.
\vspace{-3mm}
\paragraph{Properties of the cost as function of curvature.}
Given a geodesically uniquely convex subset $A \subset M$ and $p \in A$, we consider two points $x,y \in A$. We are interested in bounding distances in the geodesic triangle $\Updelta xyp$. Corollary \ref{cor:cosinelaw} states that
\begin{align*}
    d(x,p)^2 \geq \delta d(x,y)^2+d(y,p)^2-2 \langle \textnormal{log}_y(p), \textnormal{log}_y(x) \rangle
\end{align*}
Taking into consideration that the gradient of the function $f(x)=d(x,p)^2$ is $\textnormal{\textnormal{gradf}}(x)=-2 \textnormal{log}_x(p)$, the last inequality is equivalent to
\begin{equation*}
    f(x) \geq f(y)+ \langle \textnormal{\textnormal{gradf}}(y), \textnormal{log}_y(x) \rangle+\frac{2 \delta}{2} \| \textnormal{log}_x(y) \|^2
\end{equation*}
As shown in the appendix, this inequality is tight in the spherical case. This inequality also means that $f$ is either geodesically $2 \delta$-strongly convex, convex (but not strongly-convex) or not convex,
if $\delta>0$, $\delta=0$, or $\delta<0$ respectively. The first case happens, when $d(x,p) < \frac{\pi}{2 \sqrt{K_{\max}}}$, the second when $d(x,p) = \frac{\pi}{2 \sqrt{K_{\max}}}$ and the third when $ \frac{\pi}{2 \sqrt{K_{\max}}}<d(x,p)<\frac{\pi}{\sqrt{K_{\max}}}$.
\newline
However, note that the function $f$ is always 1-weakly-quasi-convex with respect to its global minimizer $p$. Indeed, from the definition $f(x) = d(x,p)^2$, we have
$f(x)-f(p)=\| \textnormal{log}_x(p) \|^2$ and $-\langle \textnormal{gradf}(x), \textnormal{log}_x(p) \rangle = 2 \langle \textnormal{log}_x(p), \textnormal{log}_x(p) \rangle  = 2 \| \textnormal{log}_x(p) \|^2$, which combined gives us $f(x)-f(p) \leq -\langle \textnormal{gradf}(x), \textnormal{log}_x(p) \rangle.$
\vspace{-3mm}
\paragraph{Example for a sphere.}
Consider a manifold $M$ as a sphere with constant curvature $K$. As a geodesically uniquely convex domain $A$, we take the ball $B_r(p)$ centered at $p \in A$ and with radius $r$. If $r < \frac{\pi}{2\sqrt{K}}$, then $\delta>0$, while if $r=\frac{\pi}{2\sqrt{K}}$ (i.e. $A$ is an open hemisphere), then $\delta=0$. The problem of minimizing $f(x)=d(x,p)^2$ is therefore either geodesically strongly-convex or geodesically convex depending on the value of $r$.
% If the radius of $A$ is exactly equal to $\frac{\pi}{2\sqrt{K_{\max}}}$ (thus $A$ is an open hemisphere), then $\delta=0$ and the function $f(x)=d(x,p)^2$ is geodesically convex, but not strongly convex (the inequality is tight).
Alternatively, if we choose to construct our geodesically uniquely convex domain $A$ as an open hemisphere with $p \in A$ not at the center, then there are points with distance from $p$ more than $\frac{\pi}{2\sqrt{K}}$. Thus $\delta$ is negative and $f$ is not geodesically convex. Given that $f(x)=d(x,p)^2$ is always $1$-weakly-quasi-convex, the problem of minimizing $f$ is weakly-quasi-convex but not convex.
\vspace{-3mm}
\paragraph{Duality smoothness/convexity.}
Lemma 5 in \cite{zhang2016first} states that the  function $f(x)=d(x,p)^2$ is $2\zeta$-smooth. This shows that there is some sort of duality between convexity and smoothness with respect to the curvature of the manifold. For a given function $d(x,p)^2$, a smaller curvature makes the function more convex while also making it less smooth.

%%%%%%%%%%%%%%%%%%%%%
%%%%%%%%%%%%%%%%%%%%%%%%%%%%%%%%%%%%%%

\section{Accelerated flows}
\label{sec:accelerated_ode}
Recall that the problem that we investigate is minimizing a function $f: M \to \R$. A fundamental algorithm to solve this problem is Riemannian gradient descent (RGD), which takes the form $x_{k+1} = \textnormal{exp}_{x_k}(-\eta \textnormal{gradf}(x_k))$, where $\eta>0$ is the so-called learning rate. The convergence properties of this method, extensively explored in~\cite{zhang2016first}, can be successfully studied~(see~\cite{munier2007steepest} and the appendix) by the means of its continuous-time limit $\dot X+\textnormal{\textnormal{gradf}}(X)=0$. 

In contrast, we are not aware of any prior work investigating the continuous-time formulation of an accelerated method.
Hence, taking inspiration from the seminal work of Su et al.~\cite{su2014differential}, we consider the following differential equation to model acceleration:
%~\footnote{For convenience, we do not explicitly write the dependency of $X$ on $t$ and we denote by $\dot{X} = dX/dt$ the time derivative.}
\begin{equation*}
\boxed{
\tag{RNAG-ODE}
    \nabla \dot X+ c \dot X+\textnormal{gradf}(X)=0}
\end{equation*}
For the convex and weakly-quasi-convex cases, we choose $c := c(t)=\frac{v}{t}$, where $v$ is a constant to be determined later. From now on, we define $\zeta$ as
\begin{equation*}
\zeta:= 
    \begin{cases}
    \sqrt{-K_{\min}} D \coth(\sqrt{-K_{\min}} D) &, K_{\min}<0 \\
    1 &, K_{\min}\geq 0
    \end{cases}
    \end{equation*}
where $D$ is an upper bound for the working domain. Next, following~\cite{zhang2018towards}, we make the following set of assumptions, which we will keep for the rest of the paper.

\paragraph{Assumptions}
Given $A \subseteq M$, and $f: M \to \R$,
\vspace{-4mm}
\begin{itemize}
    \setlength\itemsep{0mm}
    \item[1.] The sectional curvature $K$ inside $A$ is bounded from below, i.e. $ K \geq K_{\min}$.
    \item[2.] $M$ is a complete manifold, such that any two points are connected by some geodesic.
    \item[3.] $A$ is a geodesically uniquely convex subset of $M$, such that $\textnormal{diam}(A) \leq D$. The exponential map is globally a diffeomorphism.
    \item[4.] $f$ is geodesically $L$-smooth and all its minima are inside $A$.
    \item[5.] We have granted access to oracles which compute the exponential and logarithmic maps as well as the Riemannian gradient of $f$ efficiently.
    \item[6.] All the solutions of our derived differential equations remain inside $A$.
\end{itemize}
%\label{ass:main_assumption}
%\end{Assumption}

Note that the first four assumptions are standard in Riemmanian optimization~(\cite{munier2007steepest, zhang2016first, zhang2018towards}). The fifth assumption is mostly required for computational purpose. The last assumption could potentially be relaxed by relying on a barrier function or a projection step.

\subsection{Existence of a solution}
For strongly-convex functions, we will choose $c(t)$ to be constant, in which case existence and uniqueness of the solution can be shown to hold globally due to completeness of $M$.
\newline
When $c(t)=\frac{v}{t}$, the proof is not as simple and involves the use of the Arzela-Ascoli theorem for sequences of curves on Riemannian manifolds, in a similar vein as in~\cite{su2014differential}.  However,
we cannot guarantee the uniqueness of the solution.
 The proof is provided in the appendix.
 \vspace{-2mm}
\begin{restatable}{lemma}{existence}
The differential equation
\begin{equation}
    \nabla \dot X+\frac{v}{t} \dot X+\textnormal{\textnormal{gradf}}(X)=0
\label{eq:ODE_convex}
\end{equation}
where $v$ is a positive constant, has a global solution $X:[0, \infty) \rightarrow M$ under the initial conditions $X(0)=x_0 \in A$ and $\dot X(0)=0$.
\end{restatable}

The proof relies on the following result that might be of independent interest and is close to the fundamental theorem of calculus for vector fields on Riemannian manifolds.

\begin{restatable}{lemma}{meanvalue}
\label{mean value manifolds}
Consider a vector field $A$ along the smooth curve $X:[a,b] \rightarrow M$ in a 
 Riemannian manifold $M$. Then
\vspace{-3mm}
\begin{equation*}
    \Gamma_{X(b)}^{X(a)} A(b)- A(a)= \int_a^b \Gamma_{X(t)}^{X(a)} \nabla A(t) dt
\end{equation*}
where $\Gamma$ is the parallel transport along the curve $X$.
\end{restatable}

\vspace{-3mm}
\subsection{The convex case}
Now we are ready to analyze the convergence rate of the solutions of Eq.~\ref{eq:ODE_convex}, starting from a point $X(0) \in A$, to a minimizer $x^*$ of a geodesically convex function $f$.
\begin{tcolorbox}
\begin{restatable}{theorem}{convex}
Let $f$ be a geodesically convex function.
Any solution of the differential equation
\begin{equation}
\nabla \dot X+\frac{1+2 \zeta}{t} \dot X+\textnormal{\textnormal{gradf}}(X)=0
\end{equation}
converges to a minimizer $x^*$ of $f$ with rate
\begin{equation*}
\label{eqn:convex}
    f(X)-f(x^*) \leq \frac{2\zeta \| \textnormal{\textnormal{log}}_{x_0}(x^*) \|^2}{t^2} \quad (t > 0).
\end{equation*}
\vspace{-6mm}
\label{thm:thm_cvx}
\end{restatable}
\end{tcolorbox}
\vspace{-2mm}
\begin{proof}[Proof sketch]
The proof is done by showing that the following Lyapunov function is decreasing:
\begin{align*}
    \epsilon(t)=t^2(f(X)-f(x^*))+2\| -\textnormal{\textnormal{log}}_X(x^*)+\frac{t}{2} \dot X \|^2 &\\ + 2(\zeta-1) \| \textnormal{\textnormal{log}}_X(x^*) \|^2.
\end{align*}
The novelty compared to~\cite{su2014differential} is the last curvature-dependent summand. Complete proof in the appendix.
\end{proof}
\vspace{-5mm}
\subsection{The weakly-quasi-convex case}
For $\alpha$-weakly-quasi convex functions, we have the following result.
\begin{tcolorbox}
\begin{restatable}{theorem}{weaklyquasiconvex}
Let $f$ be a geodesically $\alpha$-weakly-quasi-convex function. Any solution of the differential equation
\begin{equation}
\label{eqn:quasiconvex}
    \nabla \dot X+\frac{1+\frac{2}{\alpha} \zeta}{t} \dot X+\textnormal{\textnormal{gradf}}(X)=0
\end{equation}
converges to a minimizer $x^*$ of $f$ with rate
\begin{equation*}
    f(X)-f^* \leq \frac{2 \zeta \| \textnormal{log}_{x_0}(x^*) \|^2}{\alpha^2 t^2} \quad (t > 0).
\end{equation*}
\label{thm:thm_wqc}
\end{restatable}
\vspace{-7mm}
\end{tcolorbox}
The proof is similar to the one of the convex case and can be found in the appendix. Note here that $\alpha$ can be larger than $1$. An important specific case is the Riemannian squared distance $d(x,p)^2$, where $\alpha=2$.
\vspace{-3mm}
\subsection{The strongly-convex case}
Recall that we have a constant friction term for strongly-convex functions, which yields an ODE similar to Equation 7 in~\cite{wilson2016lyapunov} for the Euclidean case.
\begin{tcolorbox}
\begin{restatable}{theorem}{stronglyconvex}
Let $f$ be a geodesically $\mu$-strongly convex function. The solution of the differential equation
\begin{equation}
\label{eqn:stronglyconvex}
   \nabla \dot X+\left(\frac{1}{\sqrt \zeta}+ \sqrt \zeta\right) \sqrt{\mu} \dot X+\textnormal{\textnormal{gradf}}(X)=0 
\end{equation}
converges to a minimizer $x^*$ of $f$ with rate
\begin{equation*}
f(X)-f^* \leq \frac{\frac{\mu}{2} \| \textnormal{\textnormal{log}}_{x_0}(x) \|^2 +f(x_0)-f^*}{e^{\sqrt{\frac{\mu}{\zeta}}t}} \quad (t > 0).   
\end{equation*}
\label{thm:thm_sc}
\end{restatable}
\vspace{-6mm}
\end{tcolorbox}
\vspace{-2mm}
\begin{proof}[Proof sketch]
The proof (see appendix) shows that the following energy function is monotically decreasing:
\vspace{-2mm}
\begin{multline*}
    \epsilon(t) = e^{\sqrt{\frac{\mu}{\zeta}}t}\Big(\frac{\mu}{2 \zeta} \| -\textnormal{\textnormal{log}}_X(x^*)+\sqrt{\zeta/\mu} \dot X \|^2\\  f(X)-f^*+\frac{\mu (\zeta-1)}{2 \zeta} \| \textnormal{\textnormal{log}}_X(x^*) \|^2 \Big).
\end{multline*}
\vspace{-3mm}
\end{proof}
\vspace{-3mm}
Note that the constant $\sqrt{\zeta}+\frac{1}{\sqrt{\zeta}}$ is always greater or equal than $2$ and equality holds only when $\zeta=1$, in which case we recover the Euclidean formulation.

\subsection{Comparison to the Euclidean case}
Compared to the ODE derived in~\cite{su2014differential}, the second derivative of the curve $X$ has been substituted with the covariant derivative of the vector field $\dot X$. This is the usual intrinsic way to capture second order changes on manifolds. The Lyapunov functions chosen in the convergence analysis are such that the covariant derivative arises when taking its derivative, which explains why the results derived in Section~\ref{sec:Hessian_distance_function} are needed in our analysis.
Also interesting is the effect of the curvature: we note that it is involved in both the friction term of the ODE and in the convergence rates. The positive-curvature case matches the Euclidean one, while the negative-curvature case yields worse constants in terms of theoretical guarantees. This seems to validate the intuition that convergence is easier in spaces with larger curvature, which is also consistent with the results of~\cite{zhang2016first}.

%%%%%%%%%%%%%%%%%%%%%%%%%%%%
%%%%%%%%%%%%%%%%%%%%%%%%%%%%%%%%%%%%%%

\section{Discretization}
In this section, we design and test a Nesterov-inspired semi-implicit integration scheme that translates the ODEs above into implementable accelerated optimization methods. Starting from the general ODE $\nabla \dot X + \alpha(t) \dot X + \textnormal{gradf}(X)=0$ and following the Euclidean modus operandi \cite{shi2019acceleration,betancourt2018symplectic}, our first step is to introduce a velocity variable $V = \dot X$. Hence, we can write $\nabla V = -\alpha(t) V -\textnormal{gradf}(X)$.

The semi-implicit Euler method in Euclidean spaces is a numerical integrator tailored to second-order ODEs, which leverages on the velocity/position decomposition and is widely used in physics because of its energy and volume conservation properties, that in turn imply good stability and small integration errors~\cite{hairer2006geometric}. This scheme consists of a standard forward-Euler update on the velocity variable $v_k$, followed by an update on the position variable $x_k$ using the just updated value of the velocity, i.e. $v_{k+1}$. Namely, if $M = \R^d$, we have
\vspace{-1mm}
\begin{equation}
    \begin{cases}
         v_{k+1} = \beta_k v_k -h \nabla f(x_k)\\
         x_{k+1} = x_k +h v_{k+1}
    \end{cases}
    \label{eq:SIE_Euclidean}
\end{equation}
where $\beta_k := 1-h\alpha(kh)$ is the momentum parameter and $h$ is the integration step-size which, if small enough, guarantees \footnote{For a \textit{fixed} interval $[0,T]$ with $T=Kh$ ($K\in\mathbb{N}$), we have $\|X(kh)-x_k\|=O(h)$ for all $0\le k\le K$~~\cite{hairer2006geometric}. However, the notation hides an exponential dependency on $T$: i.e., does not imply shadowing (see Section \ref{sec:shadowing}).} $X(kh)\approxeq x_k$. Inspired from the recent success of similar integrators in yielding accelerated algorithms~\cite{shi2019acceleration,maddison2018hamiltonian}, we provide a simple adaptation of the semi-implicit method to the Riemannian setting in the next lines. 

\begin{algorithm}
\caption{SIRNAG}
\label{alg:acc}
\begin{algorithmic}[1]
\STATE $x_0 \gets $ random point on $M$;
\STATE $v_0 \gets 0\in T_{x_0}M$;
\STATE $h \ \gets $ some small number $>0$ (integration step);
\IF {geod.~strongly-convexity} \STATE $\beta_k \gets 1-h\frac{(1+\zeta)\sqrt\mu}{\sqrt\zeta};$
\ELSIF{geod.~weak-quasi-convexity}
\STATE $\beta_k \gets \frac{k-1}{k+2\zeta/\alpha};$
\ENDIF
\FOR {$k \geq 0$}
      \STATE \textbf{Option I}: \ $a_k \gets \beta_k v_{k} - h \textnormal{gradf}(x_k)$;
      \STATE \textbf{Option II}: $a_k \gets \beta_k v_{k} - h\textnormal{gradf}(\textnormal{exp}_{x_k}(h \beta_k v_k))$
      \STATE $x_{k+1} \gets \textnormal{exp}_{x_k}(h a_k)$;
      \STATE $v_{k+1} \gets \Gamma_{x_k}^{x_{k+1}}a_k$;
\ENDFOR
\end{algorithmic}
\end{algorithm}

We start by noting that, since we require $v_k \in T_{x_k}M$ for all $k$, our method will have to include parallel transport of velocity vectors along the geodesics of the manifold. However, we can postpone this operation to the very end: indeed, if we let $a_k := \beta_k v_k - h\textnormal{gradf}(x_k)$, then $a_k\in T_{x_k}M$ and we can update the position directly using a forward-Euler step: $x_{k+1} = \textnormal{exp}_{x_k}(h a_k)$. To conclude, we need to transport the just used velocity $a_k$ to $T_{x_{k+1}}M$: $v_{k+1} = \Gamma_{x_k}^{x_{k+1}} a_k$. 

We summarize the content of the last lines in Algorithm~\ref{alg:acc} (with Option I) and provide a variant (Option II), inspired by the reformulation of Nesterov's method provided by~\cite{sutskever2013importance}. In this seminal paper the authors showed that Equation~\eqref{eq:SIE_Euclidean} is exactly Nesterov's method~\cite{nesterov2018lectures} once we replace $\nabla f(x_k)$ with $\nabla f(x_k + h \beta_k v_k)$ (the so-called \textit{corrected gradient}). In our setting, we can similarly use $\textnormal{gradf}(\textnormal{exp}_{x_k} (h \beta_k v_k))$. \textbf{As a result, Algorithm~\ref{alg:acc} with Option II \textit{reduces to Nesterov's method} when $M=\R^d$.}
 
\begin{figure}
\centering
\includegraphics[width=0.82\linewidth]{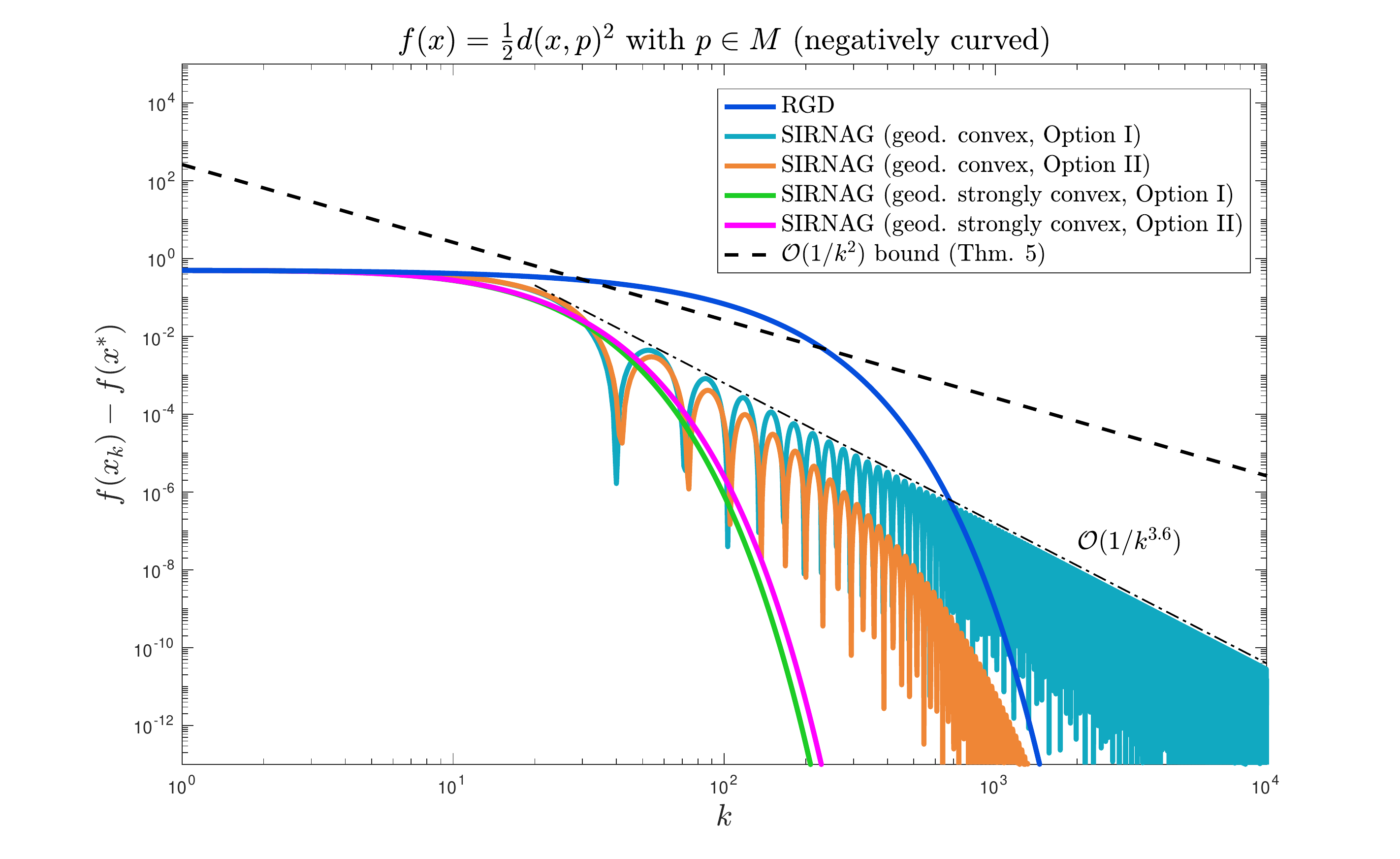}
\vspace{-2mm}
\caption{\small Dynamics of SIRNAG ($h = 0.1$) and RGD ($\eta = h^2$, see footnote~\ref{note_step}) on a subset of diameter $D=1$ of the hyperbolic space $M=\mathbb{H}^2$ ($K=-1$, hence $\zeta = \coth(1) \approxeq 1.313$) equipped with the convex (actually, strongly convex) toy  function $f(x) = \frac{1}{2} d(x,p)^2$ for $p\in M$. Plotted is also the bound found in Theorem~\ref{thm:thm_cvx}, discretized.}
\label{fig:hyperbolic_result}
\end{figure}
\vspace{-3mm}
\paragraph{Experiments.}Inspired by the relevance of hyperbolic geometry in machine learning~\cite{zhang2018direct,sra2015conic}, we start our empirical study by illustrating some properties of SIRNAG on manifolds with constant \textit{negative curvature}. Figure~\ref{fig:hyperbolic_result} shows that our integrator is stable and can achieve, on simple functions, a rate that is actually faster than the prediction of Theorem~\ref{thm:thm_cvx}, in perfect agreement with previous observations for similar costs in the Euclidean setting~\cite{zhang2018direct,betancourt2018symplectic}. Moreover, as expected, Option II provides a speed-up\footnote{Actually Option II in the geodesically strongly-convex case seems a bit slower. This happens because $f$ is of a very particular form, and is well known in the Euclidean literature (see e.g. Proposition 1 in~\cite{lessard2016analysis}).} over Option I because it is closer to the original Nesterov's method. Next, to test the tightness of the oracle bound provided by Theorem~\ref{thm:thm_cvx}, we use our algorithm to solve a high-dimensional eigenvalue problem. Indeed, the leading unit eigenvector of a symmetric matrix $Q \in \mathbb{R}^{m\times m}$ maximizes $x^T Q x$ over the unit sphere $M = \mathbb{S}^{m-1}$ (constant \textit{positive curvature}). It is well known~\cite{dieuleveut2017harder} that such objectives, when $M=\mathbb{R}^m$, are hard to optimize if $Q$ is high-dimensional and ill-conditioned, and are therefore able to truly showcase the acceleration phenomenon\footnote{Indeed, high dimensional quadratics are used to build lower bounds in convex optimization~\cite{nesterov2018lectures}.} for convex but not necessarily strongly convex functions. Figure~\ref{fig:sphere_result} shows that this fact translates to the manifold setting: indeed, the suboptimality of SIRNAG decays as $1/k^2$ --- as predicted by our continuous-time analysis --- in contrast to RGD\footnote{\label{note_step}For RGD we used a stepsize (i.e. a gradient multiplication factor) $\eta\le 1/\lambda_{\text{max}}(Q)$, where $\lambda_{\text{max}}(Q)$ is the maximum eigenvalue of $Q$. This is the standard choice in the Euclidean setting, also motivated by the results in~\cite{zhang2016first}. To get the same gradient multiplication factor and correspondence with the optimal parameters is Nesterov's method, in SIRNAG we choose $h = \sqrt{1/\lambda_{\text{max}}(Q)}$. For further details, we direct the reader to the first few pages of~\cite{su2014differential}.} which behaves like $\mathcal{O}(1/k)$. To conclude, as an ultimate test for our discretization procedure, we verify the convergence of SIRNAG to NAG-ODE as $h\to0$ in Figure~\ref{fig:approx}.\\
The code to reproduce the experiments above is available online\footnote{ \scriptsize\url{https://github.com/aorvieto/riemann-continuous.git}}.

\begin{figure}
\centering
\includegraphics[width=0.7\linewidth]{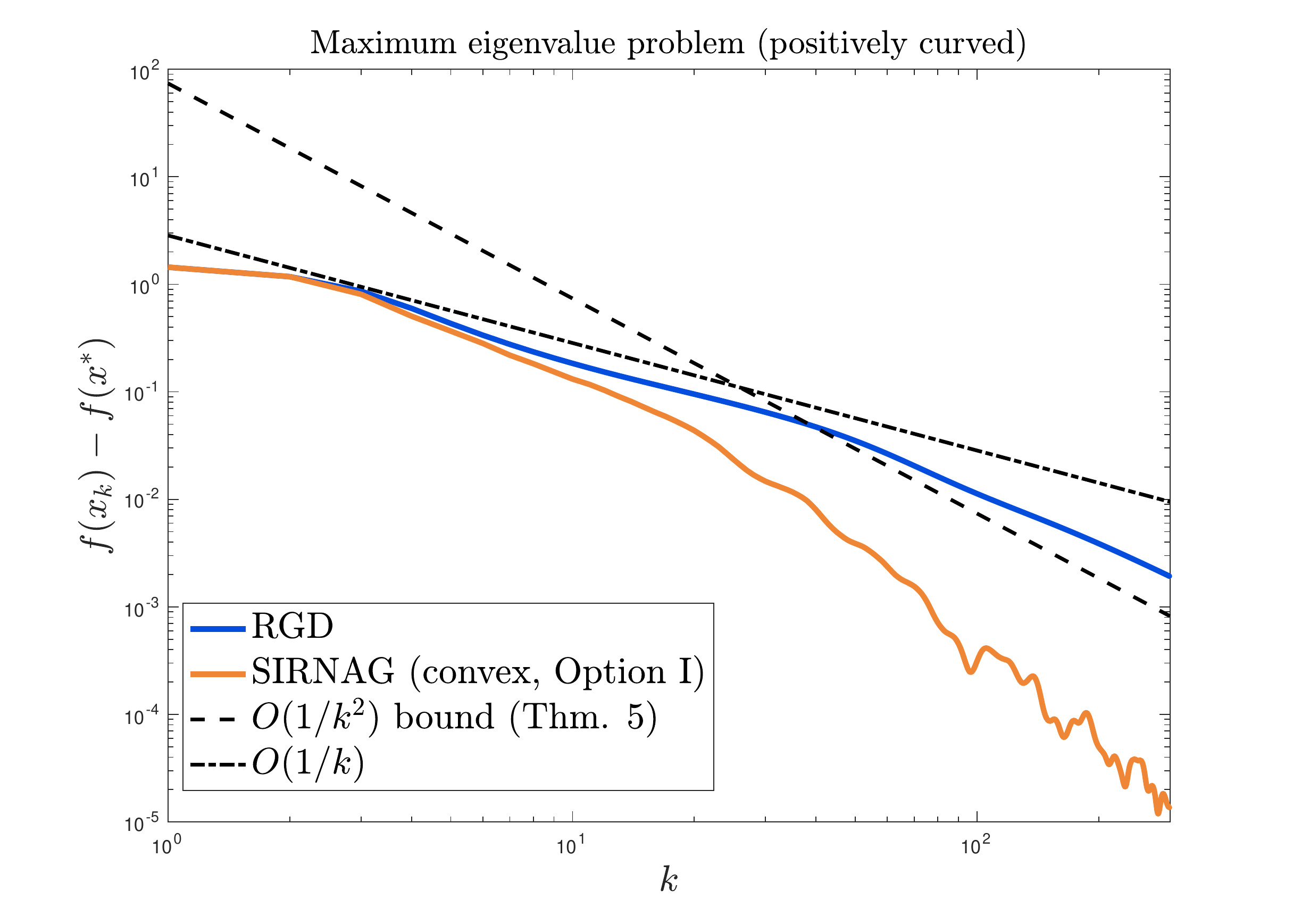}
\vspace{-2mm}
\caption{\small Performance of SIRNAG (convex, i.e. $\beta_k = \frac{k-1}{k+2\zeta}$) against RGD in finding the maximum eigenvalue of a $5$-thousand dimensional ill-conditioned matrix. Plotted is also the bound found in Theorem~\ref{thm:thm_cvx}, discretized.}
\label{fig:sphere_result}
\end{figure}

\begin{figure}
\centering
\includegraphics[width=0.8\linewidth]{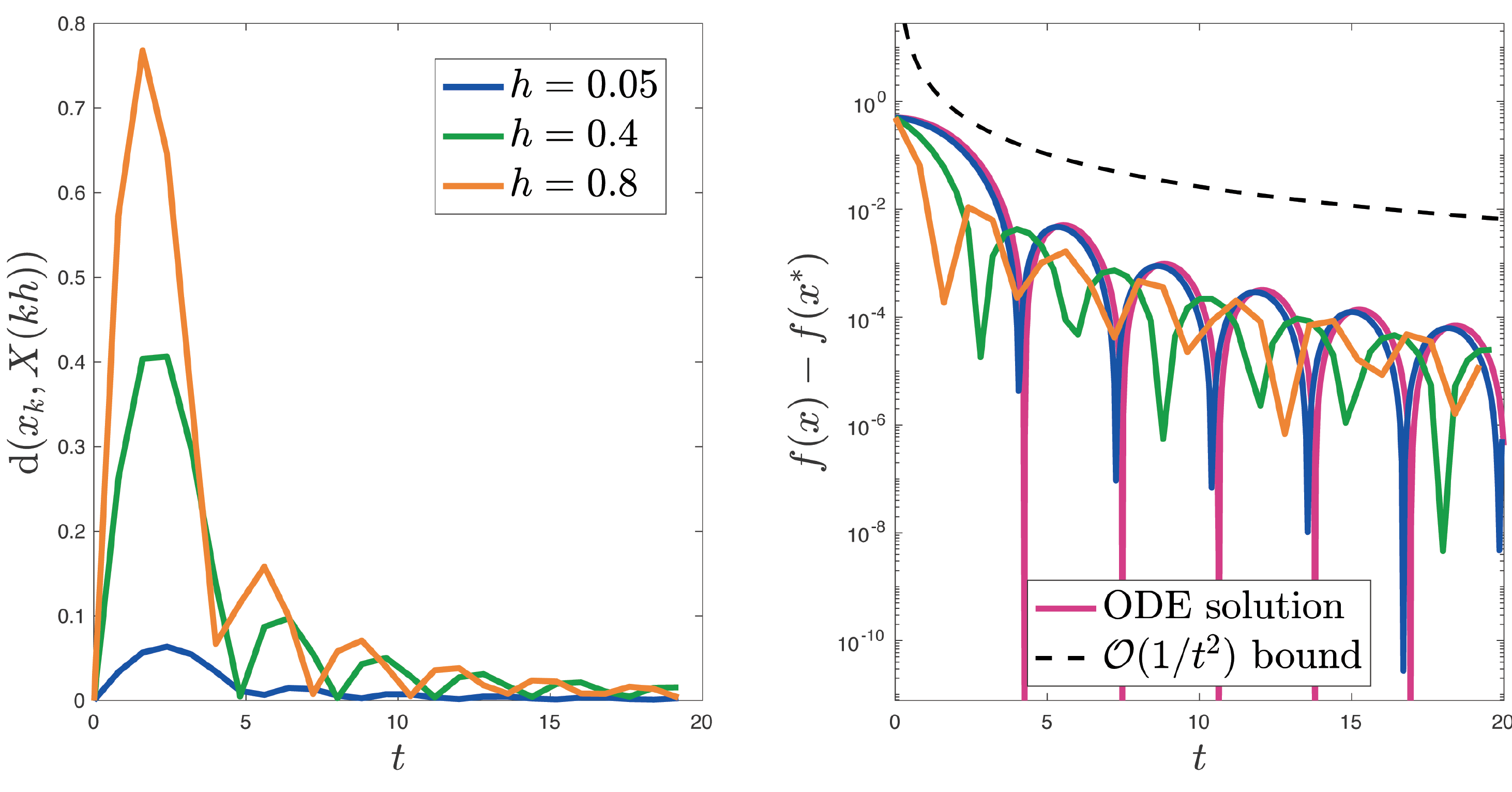}
\vspace{-2mm}
\caption{\small Convergence of SIRNAG (Option I) to the solution of Equation~\eqref{eq:ODE_convex}, same settings as Figure~\ref{fig:hyperbolic_result}. Solution to the ODE approximated by SIRNAG (Option I) with an extremely small integration step: $h = 10^{-5}$. The error peak is proportional to the step-size~(see next section).}
\label{fig:approx}
\end{figure}

%%%%%%%%%%%%%%%%%%%%%%%%%%%%%%%
%%%%%%%%%%%%%%%%%%%%%%%%%%%%%%%%%%%%%%%%%%%%%%%%%%%%%%

\vspace{-2mm}
\section{Shadowing in model spaces}
\label{sec:shadowing}
\vspace{-2mm}

So far, we have shown that the discretization of our second-order ODE empirically exhibits an accelerated rate of convergence and follows the continuous-time limit. The reader might wonder whether any theoretical guarantee can be established to bound the error between the continuous-time and discrete-time process (i.e. predict the results of Figure~\ref{fig:approx}). In the following, we will show that such guarantees can be obtained for a descent method such as RGD when compared to its limiting ODE~(studied in~\cite{munier2007steepest}). Further, in the next section, we discuss why the extension to accelerated methods is non-trivial. We will rely on the shadowing lemma for metric spaces~\cite{ombach1993simplest,brin2002introduction} and use the contraction property of RGD, as well as common concepts from the theory of dynamical systems~\cite{brin2002introduction}. We briefly review the required definitions and we refer the reader to~\cite{Antonioshadowing} for detailed explanations.
We consider a dynamical system on a Riemannian manifold $M$, i.e. a map $\Psi: M \to M$.
\begin{definition}
A sequence $(x_k)_{k=0}^\infty$ is an \textbf{orbit} of $\Psi$ if, for all $k\in\N$, $x_{k+1}=\Psi(x_{k}).$
\end{definition}

\begin{definition}
A sequence $(y_k)_{k=0}^\infty$ is a \textbf{$\boldsymbol{\delta}-$pseudo-orbit} of $\Psi$ if, for all $k\in\N$,   $d(y_{k+1},\Psi(y_{k}))\le \delta.$ 
\label{def:pseudo_orbit}
\end{definition}
%If $(y_k)_{k=0}^\infty$ is locally similar to an orbit of $\Psi$ (i.e. it is a pseudo-orbit of $\Psi$), then we might hope that such similarity extends globally. This is captured by the concept of shadowing.
\begin{definition}
A pseudo-orbit $(y_k)_{k=0}^\infty$ of $\Psi$ is \textbf{$\boldsymbol{\epsilon}-$shadowed} if there exists an orbit $(x_k)_{k=0}^\infty$ of $\Psi$ such that, for all $k\in\N$, $d(x_k,y_k)\le \epsilon$.
\label{def:shadowing}
\end{definition}
In this section, we pick $\Psi$ to be the dynamical system associated with Riemannian gradient descent, which maps $x$ to $\textnormal{exp}_x(-h \textnormal{gradf}(x))$. Its orbit $(x_k)_{k=0}^\infty$ is a sequence of iterates returned by RGD. As a candidate pseudo-orbit, we pick $(y_k)_{k=0}^\infty$ to the sequence of points derived from the iterative application of $\varphi_h$ --- the time-$h$ flow of the ODE $\dot y=-\textnormal{gradf}(y), y(0)=y_0\in M$), which is itself a dynamical system. The latter sequence represents our ODE approximation of the algorithm $\Psi$. Our goal in this subsection is to show that, under some conditions, the sequence $(y_k)_{k=0}^\infty$ is \textit{close to} an orbit of $\Psi$, uniformly in $k$ --- i.e. that it is shadowed by $\Psi$. To prove this result, we need a fundamental lemma.

\begin{lemma}(Contraction map shadowing~\cite{ombach1993simplest})
   Assume that $\Psi$ is uniformly contracting with constant $0<\rho<1$. Then, for every $\epsilon > 0$, there exists $\delta > 0$ such that every $\delta$-pseudo-orbit $(y_k)_{k=0}^{\infty}$
of $\Psi$ is $\epsilon$-shadowed by the orbit
$(x_k)_{k=0}^{\infty}$ of $\Psi$ starting at $x_0 = y_0$. Moreover, $\delta \leq (1 - \rho) \epsilon$.
\end{lemma}

To use this result, we first need to prove that the ODE orbit $(y_k)_{k=0}^\infty$ is actually a pseudo-orbit of $\Psi$. This result is standard in numerical analysis, and can be also found (in a less general form) as Proposition 2 in~\cite{absil2012projection}.
We assume, in analogy with~\cite{Antonioshadowing}, that $f:M \rightarrow \mathbb{R}$ is a $C^2$ function such that\footnote{This easily holds if $f$ is geodesically $\mu$-strongly convex and $L$-smooth. In this case, $\ell$ depends on the initial condition $y_0$ and on $L$.} for all points on the ODE solution, $\| \textnormal{gradf}(x) \| \leq \ell$ and $\mu \leq \| \textnormal{Hessf}(x) \| \leq L$.

\begin{restatable}{Proposition}{pseudoorbit}There exists a constant $C$, independent of $h$ but dependent on $\ell,L$ and the Riemannian structure of $M$, such that, for any $y_0\in M$ and $k\in\mathbb{N}$,
\vspace{-2mm}
$$d(y_{k+1},\textnormal{exp}_{y_k}(-h\textnormal{gradf}(y_k))) \leq C h^2.$$
\end{restatable}
\vspace{-2mm}
Last, we need to prove that $\Psi$ is uniformly contracting. We state the result for manifolds of constant curvature $K$ and note that passing to the bounded-curvature case can be done easily by Rauch comparison theorems. 
\vspace{-2mm}
% The curvature will however be involved in the following contraction result where we show that the distance between two points $x_1, x_2 \in M$ contracts (conditionally) after one update step of gradient descent.
We start by defining the following quantities:
\begin{equation*}
\small
\zeta:=
\begin{cases} 1 &, K \geq 0  \\
     \sqrt{-K}D \coth(\sqrt{-K}D) &, K<0
     \end{cases}
\end{equation*}
\begin{equation*}
\small
     \lambda:= \begin{cases} 1 &, K \geq 0  \\
     \sinh(\sqrt{-K}D)/(\sqrt{-K}D) &, K<0
     \end{cases}
\end{equation*}

\begin{restatable}{lemma}{contraction}
\label{le:contraction}
Let $x_1,x_2 \in M$, where $M$ is a Riemannian manifold of constant curvature $K$ and $\textnormal{diam}(M) \leq D$. If $K>0$ we further assume that $D< \frac{\pi}{\sqrt{K}}$. Then, for $\xi:=\lambda(\zeta-h\mu)$ we have
\vspace{-1mm}
{\small
\begin{align*}
    d(\textnormal{exp}_{x_1}(-h \textnormal{\textnormal{gradf}}(x_1)),\textnormal{exp}_{x_2}(-h \textnormal{\textnormal{gradf}}(x_2))) \leq \xi d(x_1,x_2),
\end{align*}}
\end{restatable}
Note that, in the positive curvature case, we recover $\xi = 1-h\mu$, in analogy with the result of~\cite{Antonioshadowing}. Finally we can state our shadowing result, which is now simple application of the contraction map shadowing Lemma.
\begin{restatable}{theorem}{shadowing}
Let $\epsilon > \frac{4 C (\lambda \zeta-1)}{\lambda^2 \mu^2}$. Any orbit $(y_k)_{k=0}^{\infty}$ of Riemannian gradient flow is $\epsilon$-shadowed by an orbit $(x_k)_{k=0}^{\infty}$ of Riemannian gradient descent, given that 
$\mu>\frac{\lambda \zeta-1}{\lambda h}$
and
\vspace{-3mm}
$$h \leq \min \left \lbrace \left(\frac{\lambda \mu}{2C }+\sqrt{\frac{\lambda^2 \mu^2}{4C^2}-\frac{\lambda \zeta-1}{C \epsilon}}\right) \epsilon, \frac{1}{L} \right \rbrace.$$
\end{restatable}
In the flat and positive-curvature case $\lambda=\zeta=1$ and we recover Theorem 3 in \cite{Antonioshadowing}.

%%%%%%%%%%%%%%%%%%%%%%%%%%%%%%%
%%%%%%%%%%%%%%%%%%%%%%%%%%%%%%%%%%%%%%%%%%%%%%%%%%%%%%
\vspace{-1mm}
\section{Discussion}
\vspace{-2mm}
We proposed a second-order ODE which gives rise to a family of accelerated methods for weakly-quasi-convex and strongly-convex optimization. Using a modified semi-implicit integration scheme, we derived a cheap iterative Nesterov-inspired algorithm which is numerically stable and empirically achieves an accelerated rate of convergence for optimization problems defined over manifolds, under both positive and negative curvature. As future work, it would be desirable to establish a general shadowing theory for the second-order ODE we studied, in order to guarantee that the discretization error can be provably kept under control. As a first step towards such an ambitious goal, we derived a shadowing result for Riemannian gradient descent. We note that, as also noted by~\cite{Antonioshadowing}, the main difficulty in the construction of such a result for accelerated algorithms is the mysterious lack of contraction of momentum methods, which are notoriously non-descending and heavily oscillating. \newline
Finally, the continuous-time representation derived in this manuscript might  serve for other applications, such as analyzing the escape speed from saddle points~\cite{criscitiello2019efficiently, sun2019escaping} or for speeding-up the optimization of non-convex functions as in~\cite{carmon2017convex}.
\paragraph{Acknowledgements}
The authors would like to thank professors Urs Lang and Nicolas Boumal for helpful discussions regarding the content of this work. Gary B{\'e}cigneul is funded by the Max Planck ETH Center for Learning Systems.
%%%%%%%%%%%%%%%%%%%%%%%%%%%%%%%%%%%%%%%%

\bibliography{main}
\bibliographystyle{plain}

%%%%%%%%%%%%%%%%%%%%%%%%%%%%%%%%%%%%%%%%
% APPENDIX

\input{appendix.tex}

\end{document}

%% file: defs.tex
\newcommand{\N}{{\mathbb{N}}}
\newcommand{\R}{{\mathbb{R}}}

\newtheorem{theorem}{Theorem}

\newtheorem{lemma}[theorem]{Lemma}
\newtheorem{definition}{Definition}
\newtheorem{Proposition}[theorem]{Proposition}

%% file: appendix.tex
% !TEX root = main.tex

\newpage
\onecolumn
\appendix
\part*{Appendix}

\section{Derivative of Riemannian squared distance}
\begin{lemma}
\label{le:derivative of intrinsic distance}
Let $M$ be a Riemannian manifold and $X:I \rightarrow M$ a smooth curve and $p \in M$. Then 
\begin{equation*}
    \frac{d}{dt} d(X,p)^2= \frac{d}{dt} \| \textnormal{log}_X(p) \|^2=2 \langle \textnormal{log}_X(p), -\dot X \rangle. 
\end{equation*}
\end{lemma}
\begin{proof}
We prove firstly that if $a,b \in M$, then 
\begin{equation*}
    d(\textnormal{exp}_a)(\textnormal{log}_a(b))(\textnormal{log}_a(b))=-\textnormal{log}_b(a).
\end{equation*}
For this purpose, we consider two different parametrizations of the geodesic connecting $a$ and $b$, one starting from $a$, $\alpha(t)= \textnormal{exp}_a(t \textnormal{log}_a(b))$, and one starting from $b$, $\beta(t)= \textnormal{exp}_b(t \textnormal{log}_b(a))$. Obviously $\alpha(t)=\beta(1-t)$. Differentiating the last equation we get $d(\textnormal{exp}_a)(t\textnormal{log}_a(b))(\textnormal{log}_a(b))=-d(\textnormal{exp}_b)((1-t)\textnormal{log}_b(a))(\textnormal{log}_b(a))$. Evaluating this at $t=1$ and using that the differential of the exponential map at $0$ is the identity, the result follows. Using this result and Gauss lemma we can prove the desired result. 
\newline
Consider a curve $X:I \rightarrow M$, a point $p \in M$ and the identity
\begin{equation*}
    \textnormal{exp}_{p} (\textnormal{log}_{p}(X))=X.
\end{equation*}
Differentiating it, we get
\begin{equation*}
    d(\textnormal{exp}_{p}) (\textnormal{log}_{p}(X)) (\frac{d}{dt} \textnormal{log}_{p}(X))=\dot X.
\end{equation*}
Now we have
\begin{align*}
    \frac{d}{dt} d(X,p)^2= \frac{d}{dt} \| \textnormal{log}_{p}(X) \|^2= 2 \langle \textnormal{log}_{p}(X), \frac{d}{dt} \textnormal{log}_{p}(X) \rangle= \\ 2 \langle d(\textnormal{exp}_{p}) (\textnormal{log}_{p}(X))\textnormal{log}_{p}(X),d(\textnormal{exp}_{p}) (\textnormal{log}_{p}(X))\frac{d}{dt} \textnormal{log}_{p}(X) \rangle =
    \langle -\textnormal{log}_X(p), \dot X \rangle.
\end{align*}
The third equality follows from the fact that $d(\textnormal{exp}_{p}) (\textnormal{log}_{p}(X))$ is a radial isometry, by Gauss lemma, and the fourth by our preliminary result.
\end{proof}

\subsection{Gradient flow}
Munier proved in \cite{munier2007steepest} (Theorem 1) that the differential equation
\begin{equation*}
    \dot X=-\textnormal{gradf}(X), X(0)=x_0
\end{equation*}
has a global solution $X:[0, \infty) \rightarrow M$, given that the manifold $M$ is complete.

\subsubsection{The convex case}
\begin{theorem}
\label{th: conv grad flow}
The solution $X:[0, \infty) \rightarrow R$ of the gradient flow ODE satisfies the inequality
\begin{equation*}
    f(X)-f(x^*)\leq \frac{\| \textnormal{log}_{x_0}(x^*) \|^2}{2t},
\end{equation*}
for $t>0$.
\end{theorem}
\begin{proof}
Consider the Lyapunov function 
\begin{equation*}
    \epsilon(t)=t(f(X)-f(x^*))+\frac{1}{2} \| \textnormal{log}_X(x^*) \|^2.
\end{equation*}
We have that
\begin{equation*}
\begin{split}
\dot \epsilon(t) &= f(X)-f(x^*)+t\langle \textnormal{gradf}(X), \dot X \rangle+\langle \textnormal{log}_X(x^*), -\dot X \rangle \\ &=  f(X)-f(x^*)+t \langle \textnormal{gradf}(X), -\textnormal{gradf}(X) \rangle+\langle \textnormal{log}_X(x^*), \textnormal{gradf}(X) \rangle \\ &= (f(X)-f(x^*)+\langle- \textnormal{log}_X(x^*), \textnormal{gradf}(X) \rangle)-t \| \textnormal{gradf}(X) \|^2 \leq 0,
\end{split}
\end{equation*}
where the first equality holds due to Lemma \ref{le:derivative of intrinsic distance} and the last inequality due to geodesic convexity.
Thus, 
\begin{equation*}
    t(f(X)-f(x^*))\leq \epsilon(t) \leq \epsilon(0)=\frac{1}{2} \| \textnormal{log}_{x_0}(x^*) \| ^2.
\end{equation*}
and the result follows.
\end{proof}

\subsubsection{The weakly-quasi-convex-case}

\begin{theorem}
If a function $f$ is geodesically $\alpha$-weakly-quasi-convex, then the global gradient flow trajectory $X:[0, \infty) \rightarrow M$ satisfies
\begin{equation*}
    f(X)-f^* \leq \frac{\| \textnormal{log}_{x_0}(x^*) \|^2}{2 \alpha t},
\end{equation*}
for $t>0$.
\end{theorem}
\begin{proof}
Consider the Lyapunov function
\begin{equation*}
  \epsilon(t)=\alpha t (f(X)-f^*)+\frac{1}{2} d(X,x^*)^2. 
\end{equation*}
which is inspired by the Lyapunov function in \cite{orvieto2018continuous}
 (end of page 22). Differentiating, using Lemma \ref{le:derivative of intrinsic distance} and $\alpha$-weakly-quasi-convexity,  we get the result.
\end{proof}

\subsubsection{The strongly convex case}

\begin{theorem}
If a function $f$ is $\mu$-strongly convex, then the gradient flow trajectory minimizes it with rate
\begin{equation*}
    f(X)-f^* \leq e^{-2 \mu t} (f(x_0)-f^*),
\end{equation*}
for $t>0$.
\end{theorem}
\begin{proof}
We just differentiate the quantity $f(X)-f^*$:
\begin{equation*}
\frac{d}{dt} (f(X)-f^*)   = -\| \textnormal{gradf}(X) \|^2 \leq -2 \mu (f(X)-f^*).  
\end{equation*}
where the inequality is an important property of strong convexity, called Polyak-Lojasiewicz condition.
Now we use Gronwall's lemma and the result follows.
\end{proof}

\section{Proofs for $\boldsymbol{\nabla \textnormal{\textbf{log}}}$ and trigonometric distance bound}
\covariant*
\begin{proof}
We have that $\textnormal{\textnormal{log}}_X(p)= \textnormal{\textnormal{gradf}}$, where $f=-\frac{1}{2}d(X,p)^2$. Indeed choose $\gamma$ smooth curve passing from $X$ in the direction of a tangent vector $a\in T_X M$:
\begin{equation*}
    d (d^2(p,X))a=\frac{d}{dt} d^2(p,\gamma)_{|t=0}=\langle 2\textnormal{\textnormal{log}}_{\gamma(0)}(p),-\dot \gamma (0) \rangle=\langle -2\textnormal{\textnormal{log}}_X(p), a \rangle.
\end{equation*}
The second equality follows from Lemma \ref{le:derivative of intrinsic distance}.
Thus we are interested in $\nabla _{\dot X} \textnormal{\textnormal{log}} _X(p)= \nabla (\textnormal{\textnormal{gradf}})$. It is convenient to view $A_f=\nabla \textnormal{\textnormal{gradf}}$ as an endomorphism which acts on vector fields. Namely $A_f(B)=\nabla_B \textnormal{\textnormal{gradf}}$ and we care for $A_f(\dot X)$. We have that 
\begin{align*}
A_f  = \nabla\operatorname{grad}(-\tfrac12 r^2)
 = 
\nabla(-r \operatorname{grad}r) =
-\operatorname{grad} r \otimes dr - r \nabla \operatorname{grad} r
= -\operatorname{grad} r \otimes dr - r A_r.
\end{align*}
where $r=d(p,;)$ and $\otimes$ is the tensor product between two vector fields.
This formulation leads us to split the vector field $\dot X$ in one part parallel to $\textnormal{gradr}$ and one orthogonal (name it $Y$). Thus
\begin{equation*}
    \dot X= m \textnormal{gradr}+Y
\end{equation*}
and we have that $(\operatorname{grad} r \otimes dr)(\textnormal{gradr})=\textnormal{gradr}$, $A_r(\textnormal{gradr})=0$ (because the integral curves of $\textnormal{gradr}$ are geodesics, so $\nabla_{\textnormal{gradr}} \textnormal{gradr}=0$), $(\operatorname{grad} r \otimes dr)(Y)=0$, thus we have to evaluate the action of $A_r$ to $Y$. We know that in the case where the sectional curvature is constant and equal to $K$, we have that
\newline
$A_r(Y)=g_r(K) \| Y \| ^2$, where
\begin{equation*}
    g_r(K)= 
\begin{cases}
1/r, & K = 0,\\
(1/R) \cot(r/R), & K= 1/R^2 >0,\\
(1/R) \coth(r/R), & K= -1/R^2<0.
\end{cases}
\end{equation*}
Applying some comparison theory we can show that $\langle A_r(Y),Y\rangle \ge g_r(K_{\max})|Y|^2$ and $\langle A_r(Y),Y\rangle \leq g_r(K_{\min})|Y|^2$, for $K_{\min}\leq K \leq K_{\max}$ (check \cite{Petersen_textbook}, Proposition 25 in page 173 for Riccati comparison theory, and \cite{Lee_textbook}, chapter 11).
Now we have that
\begin{multline*}
\langle A_f( \dot X),-\dot X \rangle= \langle -m\textnormal{gradr}-rA_r(Y), -m \textnormal{gradr}-Y \rangle\\= \| m \textnormal{gradr} \|^2+m \langle \textnormal{gradr}, Y \rangle+ mr \langle \textnormal{gradr}, A_r(Y) \rangle+r \langle A_r(Y),Y \rangle. 
\end{multline*}
We have that $\langle \textnormal{gradr},Y \rangle=0$, because $Y$ and $\textnormal{gradr}$ have been assumed to be orthogonal. Also, by the fundamental theorem of Riemannian geometry, the Levi-Civita connection satisfies
\begin{align*}
    \frac{d}{dY} \langle \textnormal{gradr}, \textnormal{gradr} \rangle= \langle \nabla_Y \textnormal{gradr}, \textnormal{gradr} \rangle+\langle  \textnormal{gradr}, \nabla_Y\textnormal{gradr} \rangle=2\langle \nabla_Y \textnormal{gradr}, \textnormal{gradr} \rangle,
\end{align*}
where $\frac{d}{dY}$ is the derivative in the direction of the vector field $Y$.
Now using that $\textnormal{gradr}=grad(r^2)^{\frac{1}{2}}$, we can prove that $\textnormal{gradr}=-\frac{\textnormal{\textnormal{log}}_X(p)}{d(X,p)}$, thus $\| \textnormal{gradr} \| ^2=1$, which means that $\langle A_r(Y), \textnormal{gradr} \rangle=\langle \nabla_Y \textnormal{gradr} , \textnormal{gradr} \rangle = \frac{1}{2}   \frac{d}{dY} \| \textnormal{gradr} \|^2$=0. Thus 
\begin{equation*}
\langle A_f( \dot X),-\dot X \rangle=\| m \textnormal{gradr} \|^2+r \langle A_r(Y),Y \rangle=m^2+r \langle A_r(Y),Y \rangle
\end{equation*}
and
\begin{align*}
    \| \dot X \|^2= \langle m\textnormal{gradr}+Y, m\textnormal{gradr}+Y \rangle= \| m \textnormal{gradr} \| ^2+2 \langle m\textnormal{gradr}, Y \rangle + \| Y \|^2= m^2+\| Y \|^2.
\end{align*}
Using the previous comparison results we get
\begin{equation*}
 r g_r(K_{\max}) \| Y \|^2 \leq r\langle A_r(Y), Y \rangle \leq r g_r(K_{\min}) \| Y \|^2;    
\end{equation*} 
and equivalently
\begin{equation*}
 r g_r(K_{\max}) (\| \dot X \|^2-m^2) \leq r\langle A_r(Y), Y \rangle \leq r g_r(K_{\min}) (\| \dot X \|^2-m^2)    
\end{equation*}
and
\begin{equation*}
m^2+r g_r(K_{\max}) (\| \dot X \|^2-m^2) \leq \langle A_f(\dot X), \dot X \rangle \leq m^2+r g_r(K_{\min}) (\| \dot X \|^2-m^2).    
\end{equation*}
Thus,
\begin{equation*}
(1-r g_r(K_{\max}))m^2+r g_r(K_{\max}) \| \dot X \|^2 \leq \langle A_f(\dot X), \dot X \rangle \leq (1-r g_r(K_{\min}))m^2+r g_r(K_{\min}) \| \dot X \|^2.    
\end{equation*}
Now we have to evaluate $m$. It arises when projecting $\dot X$ to $\textnormal{gradr}$, so we can compute it by basic linear algebra. Namely
\begin{align*}
    m=\frac{\langle \dot X, \textnormal{gradr} \rangle}{\| \textnormal{gradr} \|}=\langle \dot X, \textnormal{gradr} \rangle=\langle \dot X, -\frac{\textnormal{\textnormal{log}}_X(p)}{\| \textnormal{\textnormal{log}}_X(p) \|} \rangle=\frac{1}{\| \textnormal{\textnormal{log}}_X(p) \|} \langle \textnormal{\textnormal{log}}_X(p), -\dot X \rangle
\end{align*}
and   
\begin{equation*}
   0\leq m^2=\frac{1}{\| \textnormal{\textnormal{log}}_X(p) \|^2} \langle \textnormal{\textnormal{log}}_X(p), -\dot X \rangle^2 \leq \| \dot X \|^2
\end{equation*}
by Cauchy-Schwarz inequality. 

If $K_{\max}>0$, then $rg_r(K_{\max})<1$, so 
$$(1-rg_r(K_{\max}))m^2 \geq 0 \text{ \ \ and \ \ }(1-r g_r(K_{\max}))m^2+r g_r(K_{\max}) \| \dot X \|^2 \geq
r g_r(K_{\max}) \| \dot X \|^2.$$
If $K_{\max}\leq 0$, then $rg_r(K_{\max}) \geq 1$, so 
$$(1-r g_r(K_{\max}))m^2+r g_r(K_{\max}) \| \dot X \|^2\geq (1-r g_r(K_{\max}))\| \dot X \|^2+r g_r(K_{\max}) \| \dot X \|^2=\| \dot X \|^2.$$ 
Thus we have overall that
\begin{equation*}
    \langle \nabla_{\dot X} \textnormal{\textnormal{log}} _X(p) , -\dot X \rangle \geq \delta \| \dot X \|^2,
\end{equation*}
because the function $x \cot(x)$ is decreasing for $x\geq 0$ and $r\leq D$.

Now we proceed to the other direction. 

If $K_{\min}>0$, then $rg_r(K_{\min})<1$, so 
$$(1-rg_r(K_{\min}))m^2 \leq (1-rg_r(K_{\min}))\| \dot X \|^2 \text{ \ \ and \ \ }(1-r g_r(K_{\max}))m^2+r g_r(K_{\max}) \| \dot X \|^2 \leq  \| \dot X \|^2.$$

If $K_{\min}\leq 0$, then $rg_r(K_{\min}) \geq 1$, thus $$(1-r g_r(K_{\min}))m^2+r g_r(K_{\min}) \| \dot X \|^2\leq r g_r(K_{\min}) \| \dot X \|^2.$$
Thus we have overall that
\begin{equation*}
    \langle \nabla_{\dot X} \textnormal{\textnormal{log}} _X(p) , -\dot X \rangle \leq \zeta \| \dot X \|^2,
\end{equation*}
because $r(t)=d(X(t),p)$. 
Combining these inequalities, the result follows.
\end{proof}
Of course the inequalities of Lemma \ref{le:result for covariant derivative} hold independently if we bound the curvature only in one direction.
\cosinelaw*
\begin{proof}
    Let $X$ be the side of $\Updelta abc$ connecting $b=X(0)$ and $c=X(1)$. Consider the function $w: \R_+ \rightarrow \R$, given by 
\begin{equation*}
    w(t)=\| \textnormal{\textnormal{log}}_{X(t)}(a) \|_{T_{X(t)} M}^2.
\end{equation*}
By Taylor's theorem we have that
\begin{align*}
\| \textnormal{\textnormal{log}}_c(a) \|^2- \| \textnormal{\textnormal{log}}_b(a) \|^2=\| \textnormal{\textnormal{log}}_{X(1)}(a) \|^2- \| \textnormal{\textnormal{log}}_{X(0)}(a) \|^2=\dot w(0)+\frac{1}{2} \ddot w(\xi),
\end{align*}
for some $\xi \in (0,1)$.
We have by Lemma \ref{le:derivative of intrinsic distance} that $\dot w=2 \langle \textnormal{\textnormal{log}}_X(a), -\dot X \rangle$, so 
\newline
$\ddot w=2 \langle \nabla \textnormal{\textnormal{log}}_X(a), -\dot X \rangle+ 2\langle \textnormal{\textnormal{log}}_X(a),-\nabla \dot X \rangle=2 \langle \nabla \textnormal{\textnormal{log}}_X(a), -\dot X \rangle$, because $X$ is a geodesic, which implies that $\nabla \dot X=0$.
Thus,
\begin{multline*}
\| \textnormal{\textnormal{log}}_b(a) \|^2- \| \textnormal{\textnormal{log}}_c(a) \|^2= 2 \langle \textnormal{\textnormal{log}}_{X(0)}(a),-\textnormal{\textnormal{log}}_b(c) \rangle+\langle \nabla \textnormal{\textnormal{log}}_{X(\xi)}(a),-\dot X(\xi) \rangle \\ \leq 2 \langle \textnormal{\textnormal{log}}_b(a),-\textnormal{\textnormal{log}}_b(c) \rangle+\langle \nabla \textnormal{\textnormal{log}}_{X(\xi)}(a),-\dot X(\xi) \rangle.   
\end{multline*}
By Lemma \ref{le:result for covariant derivative}, we know that 
\begin{equation*}
\langle \nabla \textnormal{\textnormal{log}}_{X(\xi)}(a),-\dot X(\xi) \rangle \geq \delta(\xi) \| \dot X(\xi) \|^2.
\end{equation*}
Using again that $X$ is a geodesic, we have
\begin{equation*}
    \frac{d}{dt} \| \dot X \|^2=2 \langle \dot X, \nabla \dot X \rangle=0
\end{equation*}
which means that $\| \dot X \|_{T_X M}^2$ is constant, thus $\| \dot X(\xi) \|_{T_{X(\xi)} M}^2=\| \dot X(0) \|_{T_{X(0)} M}^2
=\| \textnormal{\textnormal{log}}_b(c) \|^2$.
\newline
Thus
\begin{align*}
      \| \textnormal{\textnormal{log}}_c(a) \|^2- \| \textnormal{\textnormal{log}}_b(a) \|^2 \geq 2 \langle \textnormal{\textnormal{log}}_b(a),-\textnormal{\textnormal{log}}_b(c) \rangle+ \delta(\xi) \| \textnormal{\textnormal{log}}_b(c) \|^2 
\end{align*}
and equivalently
\begin{align*}
      \| \textnormal{\textnormal{log}}_c(a) \|^2  \geq 2 \langle \textnormal{\textnormal{log}}_b(a),-\textnormal{\textnormal{log}}_b(c) \rangle+ \delta(\xi) \| \textnormal{\textnormal{log}}_b(c) \|^2+\| \textnormal{\textnormal{log}}_b(a) \|^2. 
\end{align*}
Thus, the result follows for $q=X(\xi) \in bc$.
\end{proof}
According to the proofs of the last results, in the case that our manifold is a sphere, the inequality is tight. Namely, it holds as an equality if the geodesic $X=(bc)$ satisfies
\begin{equation*}
    \textnormal{log}_X(a) \bot \dot X.
\end{equation*}
We can always choose a geodesic triangle with this property in the sphere, thus our inequality is tight in the spherical case.

%%%%%%%%%%%%%%%%%%%%%%%%%%%%%
%%%%%%%%%%%%%%%%%%%%%%%%

\section{Proof of existence of a solution}
\existence*
\begin{proof}
The proof will be similar to the relevant result in \cite{su2014differential}(Appendix A). We start by modifying the equation in order to be defined at $0$. So, we get a family of equations of the form $\nabla \dot X+\frac{v}{\max(\delta,t)} \dot X + \textnormal{\textnormal{gradf}}(X)=0$, where $\delta$ is a positive real number and $X,\dot X$ continue to satisfy the same initial conditions. Since we have assumed that $\textnormal{\textnormal{exp}}$ and $\textnormal{\textnormal{log}}$ are defined globally on $M$, we can choose geodesically normal coordinates $\phi=\psi^{-1}$ around $x_0$ defined globally on $M$ and put $c=\phi \circ X$. The equation in geodesically normal coordinates is 

\begin{equation*}
\ddot c^k+\sum_{i,j=1}^{m} \Gamma_{ij}^k(c) \dot c^i \dot c^j+ \frac{v}{\max(\delta,t)} \dot c^k + \sum_{i=1}^{m} g^{ik} \frac{\partial (f o \psi)}{\partial x^i}(c)=0,
\end{equation*}
for $k=1,...,m$, where c(0)=$\phi(x_0)=0$ and $\dot c(0)=d\phi(x_0)\dot X(0)=0$. Since $f$ is of class $C^2$, we have that $\sum_{i=1}^{m} g^{ik} \frac{\partial (f o \psi)}{\partial x^i}(c)$ is smooth, thus also locally Lipschitz. Substituting $u=\dot c$ we get a system of first order ODEs, which defines a local representation for a vector field in the tangent bundle of $M$. The solution of such an ODE in local coordinates corresponds to an integral curve of this vector field in $TM$. Since an integral curve exists always locally ($TM$ is itself a manifold) and it is unique up to an initial condition, we conclude that our initial smoothed ODE $\nabla \dot X+\frac{v}{\max(\delta,t)} \dot X + \textnormal{\textnormal{gradf}}(X)=0$ has a unique solution locally around $0$. For more details in the correspondence of second order ODEs on a manifold $M$ with integral curves on $TM$ see \cite{Lang_textbook} (pages 96-99). Let $[0,T), T>0$ be the maximal existence interval of the solution $X_{\delta}$. We prove that this solution can actually be extended until infinity following an argument in \cite{munier2007steepest} (Theorem 1). Assume that $T< \infty$. We differentiate the function $f(X_{\delta})$: 
\begin{multline*}
\frac{d}{dt}(f(X_{\delta}(t)))  = \langle \textnormal{\textnormal{gradf}}(X_{\delta}), \dot X_{\delta} \rangle  =
\langle -\nabla \dot X_{\delta}-\frac{v}{\max(\delta,t)} \dot X_{\delta}, \dot X_{\delta} \rangle \\ =-\langle \nabla \dot X_{\delta}, \dot X_{\delta} \rangle-\frac{v}{\max(\delta,t)}\langle\dot X_{\delta}, \dot X_{\delta} \rangle =
-\frac{1}{2} \frac{d}{dt} \| \dot X_{\delta} \|^2-\frac{v}{\max(\delta,t)} \| \dot X_{\delta} \|^2.
\end{multline*}
Integrating each side and using Cauchy-Schwarz inequality for integrals, we get 
\begin{equation*}
\int_0^T \sqrt{\frac{v}{\max(\delta,t)}} \| \dot X_{\delta} \| dt  \leq \sqrt{T (f(x_0)-\inf_Mf+\frac{1}{2} (\| \dot X_{\delta}(0)\|^2 -  \inf_{[0,T)} \| \dot X_{\delta}(t) \|^2))}<\infty.
\end{equation*}
This is because $f$ has been assumed to be geodesically convex, thus bounded from below.
\newline
But we can split the integral in the left
 hand side as 
$\int_0^{\delta} \sqrt{\frac{v}{\delta}} \| \dot X_{\delta} \| dt + \int_{\delta}^T \sqrt{\frac{v}{t}} \| \dot X_{\delta} \| dt$.
If $0<\delta<T$, the first integral in the sum is finite, so the second is also finite. If $\delta \geq T$ we can proceed directly without splitting and get that 
$\int_0^T \sqrt{\frac{v}{\delta}} \| \dot X_{\delta} \| dt$ is finite. Thus, we have that  $\sqrt{\frac{v}{t_0}} \dot X_{\delta}:[\delta, T) \rightarrow M$ (for some $t_0\in (\delta,T)$ by the mean value theorem) and $\sqrt{\frac{v}{\delta}}  \dot X_{\delta}:[0,T) \rightarrow M$ are  integrable for each case respectively. This means that in each case the limit it of $X_{\delta}(t)$ exists, since $\| \int_a^T \dot X_{\delta} dt \| \leq \int_a^T \|  \dot X_{\delta} \| dt< \infty$, for $a=0$ or $\delta$, and in general belongs in the completion of $M$. Since $M$ is complete, the limit is in $M$. Thus we can extend the maximal existence interval. So, we have a contradiction.  Thus we can find an 
$X_{\delta}:[0,\infty) \rightarrow M$ to be a solution of the initial smoothed ODE and $X_{\delta}:[0,\infty) \rightarrow R^m$ its corresponding solution in local coordinates. Note that $\nabla \dot X_{\delta}$ is well-defined at $0$. Our purpose is to apply Arzela-Ascoli theorem in the family of the obtained solutions to get a solution for the initial ODE $\nabla \dot X+\frac{v}{t} \dot X + \textnormal{\textnormal{gradf}}(X)=0$. There are two types of parallel transport appearing in the proof, $\Gamma$ for the parallel transport along $X_{\delta}$ and $\Tilde{\Gamma}$ for the one along some
geodesic connecting the two points. When we have a covariant derivative, it refers to the first, while geodesic $L$-smoothness to the second. Their common characteristic is that they are both orthogonal transformations, thus they preserve lengths of vectors.
\newline
Now we proceed as follows:
\newline
\begin{enumerate}

\item  We define 
\begin{equation*}
\begin{split}
M_{\delta}(t)=\sup\left\lbrace\frac{\| \dot X_{\delta} (u) \|}{u}, u \in (0,t]\right\rbrace,
\end{split}
\end{equation*}
and note that it is finite, because 
\begin{equation*}
\begin{split}
\frac{\| \dot X_{\delta} (u) \|}{u}=\frac{\| \Gamma_{X_{\delta} (u)}^{X_{\delta} (0)} \dot X_{\delta} (u)-\dot X_{\delta} (0) \|}{u}=\| \nabla \dot X_{\delta}  (0) \| +o(1)
\end{split}
\end{equation*}
 for small $u$.
\newline
\item  We have that $\| \textnormal{\textnormal{gradf}}(X_{\delta}(u))-\Tilde{\Gamma}_{x_0}^{X_{\delta}(u)} \textnormal{\textnormal{gradf}}(x_0) \| \leq \frac{1}{2}LM_{\delta}(u)u^2$. Indeed, by Lipschitz assumption about $f$, we have that 
\begin{multline*}
    \| \textnormal{\textnormal{gradf}}(X_{\delta}(u))-\Tilde{\Gamma}_{x_0}^{X_{\delta}(u)} \textnormal{\textnormal{gradf}}(x_0) \| \\ \leq L d(X_{\delta}(u),x_0) \leq L \int_0^u \| \dot X_{\delta}(s) \| ds = L \int_0^u s \frac{\| \dot X_{\delta}(s) \|}{s} ds \leq \frac{1}{2} L M_{\delta}(u)u^2.
\end{multline*}
\item  For $\delta < \sqrt{\frac{6}{L}}$, we have that
\begin{equation*}
\begin{split}
M_{\delta}(\delta) \leq \frac {\| \textnormal{\textnormal{gradf}}(x_0)  \|}{1-\frac{L \delta^2}{6}}.
\end{split}
\end{equation*}
Indeed for $0<t\leq \delta$, we have 
\begin{equation*}
\begin{split}
\nabla \dot X_{\delta}+ \frac{v}{\delta} \dot X_{\delta} + \textnormal{\textnormal{gradf}}(X_{\delta})=0.
\end{split}
\end{equation*}
This equation can be written as 
\begin{align*}
    \nabla (\dot X_{\delta}(t) e^{\frac{vt}{\delta}})=-\textnormal{\textnormal{gradf}}(X_{\delta}) e^{\frac{vt}{\delta}}.
\end{align*}
By Lemma \ref{mean value manifolds} we have
\begin{multline*}
    \Gamma_{X_{\delta}(t)} ^ {x_0} \dot X_{\delta}(t) e^{\frac{vt}{\delta}}\\= -\int_0^t (\Gamma_{X_{\delta}(u)}^{x_0} \textnormal{\textnormal{gradf}}(X_{\delta}(u))-\Gamma_{X_{\delta}(u)}^{x_0} \Tilde{\Gamma}_{x_0}^{X_{\delta}(u)}  \textnormal{\textnormal{gradf}}(x_0)) e^{\frac{vt}{\delta}} du -\int_0^t \Gamma_{X_{\delta}(u)}^{x_0} \Tilde{\Gamma}_{x_0}^{X_{\delta}(u)}  \textnormal{\textnormal{gradf}}(x_0))e^{\frac{vt}{\delta}} du. 
\end{multline*}
Using point 2 and the fact that parallel transports $\Gamma, \Tilde{\Gamma}$ are orthogonal transformations, thus they preserve lengths, we can follow the proof of Lemma 15 in \cite{su2014differential}.
\newline
\item  For $\delta< \sqrt{\frac{6}{L}}$ and $\delta<t<\sqrt{\frac{2(v+3)}{L}}$, we have
\begin{equation*}
\begin{split}
&M_{\delta}(t) \leq  \frac{(v+2-\frac{L \delta^2}{6})\| \textnormal{\textnormal{gradf}}(x_0) \|}{(v+1) (1-\frac{L \delta^2}{6})(1-\frac{L t^2}{2(v+3)})}.
\end{split}
\end{equation*}

Indeed for $t>\delta$ the smoothed ODE is 
\begin{equation*}
\begin{split}
 \nabla \dot {X}_{\delta}+\frac{v}{t} \dot X_{\delta} + \textnormal{\textnormal{gradf}}({X}_{\delta})=0.
\end{split}
\end{equation*}
This equation is equivalent to 
\begin{equation*}
\begin{split}
 \frac{d(t^v \dot X_{\delta}(t))} {dt} &=-t^v \textnormal{\textnormal{gradf}}(X_{\delta}(t))
\end{split}
\end{equation*}
and using again Lemma \ref{mean value manifolds}, we get
\begin{multline*}
\Gamma_{X_{\delta}(t)}^{X_{\delta}(\delta)} t^v \dot X_{\delta}(t)-\delta^v \dot X_{\delta}(\delta)=\\-\int_{\delta}^t (\Gamma_{X_{\delta}(u)}^{X_{\delta}(\delta)} \textnormal{\textnormal{gradf}}(X_{\delta}(u))-\Gamma_{X_{\delta}(u)}^{X_{\delta}(\delta)} \Tilde{\Gamma}_{x_0}^{X_{\delta}(u)} \textnormal{\textnormal{gradf}}(x_0)) u^v du - \int_{\delta}^t\Gamma_{X_{\delta}(u)}^{X_{\delta}(\delta)} \Tilde{\Gamma}_{x_0}^{X_{\delta}(u)} \textnormal{\textnormal{gradf}}(x_0) u^v du. 
\end{multline*}
Rearranging, putting norms and dividing by $t^{v+1}$, we get 
\begin{equation*}
\begin{split}
&\frac{\| \dot X_{\delta}(t) \|}{t}  \leq  \frac{t^{v+1}-\delta^{v+1}}{(v+1)t^{v+1}} \| \textnormal {\textnormal{gradf}}(x_0) \|+\frac{1}{t^{v+1}} \int_{\delta}^t \frac{1}{2} L M_{\delta}(u)u^{v+2} du +\frac{\delta^{v+1}}{t^{v+1}} \frac{\| \dot X_{\delta} (\delta) \|}{\delta} \\ &\leq \frac{1}{v+1} \| \textnormal{\textnormal{gradf}}(x_0) \|+\frac{1}{2(v+3)} L M_{\delta} (t) t^2+\frac{\| \textnormal{\textnormal{gradf}}(x_0) \|}{1-\frac{L \delta^2}{6}},
\end{split}
\end{equation*}
using again that parallel transport preserve lengths.
The last expression is an increasing function of $t$, thus for any $t' \in (\delta,t)$ we have
\begin{equation*}
    \frac{\| \dot X_{\delta}(t') \|}{t'} \leq \frac{1}{v+1} \| \textnormal{\textnormal{gradf}}(x_0) \|+\frac{1}{2(v+3)} L M_{\delta} (t) t^2+\frac{\| \textnormal{\textnormal{gradf}}(x_0) \|}{1-\frac{L \delta^2}{6}}.
\end{equation*}
Taking the supremum over all $t' \in (0,t)$ and rearranging, we get the result.

\item  The family $A:=\lbrace X_{\delta}:[0,\sqrt{\frac{v+3}{L}}] \rightarrow \mathbb{R}/\delta=\frac{\sqrt{\frac{3}{L}}}{2^n}, n=0,1,... \rbrace$ is uniformly bounded and equicontinuous.
By the definition of $M_{\delta}$ we have that $\| \dot X_{\delta} \| \leq \sqrt{\frac{v+3}{L}} M_{\delta}(\sqrt{\frac{v+3}{L}})$. For $t \in [0,\sqrt{\frac{v+3}{L}}]$ and $\delta \in (0,\sqrt{\frac{3}{L}})$, we get a uniform bound for 
$\| \dot X_{\delta} \|$: 
\begin{equation*}
\begin{split}
\| \dot X_{\delta} \| \leq \sqrt{\frac{v+3}{L}} \max \left\lbrace \frac{\| \textnormal{\textnormal{gradf}}(x_0) \|}{1-\frac{1}{2}}, \frac{(v+2-\frac{1}{2}) \| \textnormal{\textnormal{gradf}}(x_0) \|}{(v+1) (1-\frac{1}{2})(1-\frac{1}{2})} \right\rbrace.
\end{split}
\end{equation*}
This implies that $A$ is equicontinuous. In addition,
\begin{equation*}
\begin{split}
& d(X_{\delta}(t),X_{\delta}(0)) \leq \int_0^t \| \dot X_{\delta}(u) \| du \leq \frac{v+3}{L} \max \left\lbrace \frac{\| \textnormal{\textnormal{gradf}}(x_0) \|}{1-\frac{1}{2}}, \frac{(v+2-\frac{1}{2}) \| \textnormal{\textnormal{gradf}}(x_0) \|}{(v+1) (1-\frac{1}{2})(1-\frac{1}{2})} \right\rbrace.
\end{split}
\end{equation*}
Thus $A$ is also uniformly bounded.
\newline
Finally we are ready to apply the Arzela-Ascoli theorem. We use a version which can be applied to Riemannian manifolds, see \cite{Kelley_textbook} (Theorem 17, page 233). We also make use of the fact that our manifold has been assumed to be complete to guarantee point (b) of the theorem.
\newline
It implies that $A$ contains a subsequence, which converges uniformly on $[0,\sqrt{\frac{v+3}{L}}]$. Let $\lbrace X_{\delta_{m_i}} \rbrace$ be this convergent subsequence and $w$ the limit. Pick a point $t_0 \in (0, \sqrt{\frac{v+3}{L}})$. Since $\| \dot X_{\delta}(t_0) \|$ is bounded, it has a convergent subsequence, which can be assumed without loss of generality to be the whole sequence. Denote by $s$ the local solution of our smoothed differential equation, such that $s(t_0)=w(t_0)$ and $\dot s(t_0)=\dot X_{\delta_{m_i}}(t_0)$, if $\delta_{m_i}<t_0$. We conclude that there exists $\epsilon_0>0$, such that $\sup \lbrace \| X_{\delta_{m_i}}(t)-s(t) \| / t_0-\epsilon_0<t<t_0+\epsilon_0 \rbrace$ tends to $0$, when $i$ goes to $\infty$. By definition of $w$, we have the same convergence for $w$ in the place of $s$. Thus $s\equiv w$ in $(t_0-\epsilon_0,t_0+\epsilon_0)$, thus they coincide also at $t_0$, therefore $w$ is a solution of the (non-smoothed) ODE at $t_0$. But $t_0$ was arbitrary, so $w$ is a solution of the (non-smoothed) ODE on $(0, \sqrt{\frac{v+3}{L}})$. We can extend $w$ until $\infty$ to get a global solution. Now it remains to verify the initial conditions. Since $X_{\delta_{m_i}}(0)=x_0$ and $X_{\delta_{m_i}}(0) \rightarrow w(0)$, we get easily that $w(0)=x_0$. For the condition of the initial velocity, we pick a small $t>0$ and consider 
\begin{equation*}
\begin{split}
\frac{d(w(t),w(0))}{t}=\lim_{i\rightarrow \infty} \frac{ d(X_{\delta_{m_i}}(t),X_{\delta_{m_i}}(0))}{t} \leq  \lim_{i \rightarrow \infty} \frac{1}{t} \int_0^t \| \dot X_{\delta_{m_i}}(u) \| du = \lim_{i\rightarrow \infty} \| \dot X_{\delta_{m_i}}(l_i) \|,
\end{split}
\end{equation*}
where $l_i \in (0,t)$ is obtained by the mean value theorem. By the definition of $M_{\delta}$, we get that the left hand side is less or equal than 
\begin{equation*}
\begin{split}
\limsup_{i\rightarrow \infty} \ t M_{\delta_{m_i}}(t) \leq t \sqrt{\frac{v+3}{L}} \max \left\lbrace \frac{\| \textnormal{\textnormal{gradf}}(x_0) \|}{1-\frac{1}{2}}, \frac{(v+2-\frac{1}{2}) \| \textnormal{\textnormal{gradf}}(x_0) \|}{(v+1) (1-\frac{1}{2})(1-\frac{1}{2})} \right\rbrace.
\end{split}
\end{equation*}
Sending $t$ to $0$, we get $\dot w(0)=0$ and we are done.
\end{enumerate}
\end{proof}
\meanvalue*
\begin{proof}
Consider the function $g:[a,b] \rightarrow T_{X(a)} M$, defined by
\begin{equation*}
    g(t)= \Gamma_{X(t)}^{X(a)} A(t).
\end{equation*}
$T_{X(a)} M$ is a linear space, and we have
\begin{equation*}
    g(b)-g(a)= \int_a^b \dot g(t) dt.
\end{equation*}
We have that $g(b)=\Gamma_{X(b)}^{X(a)} A(b)$, $g(a)=A(a)$ and
\begin{align*}
    \dot g(t) &= \lim_{s \rightarrow t} \frac{g(s)-g(t)}{s-t}=\lim_{s \rightarrow t}\frac{\Gamma_{X(s)}^{X(a)} A(s)-\Gamma_{X(t)}^{X(a)} A(t)}{s-t}= \Gamma_{X(t)}^{X(a)} \lim_{s \rightarrow t}\frac{\Gamma_{X(a)}^{X(t)} \Gamma_{X(s)}^{X(a)} A(s)- A(t)}{s-t} \\ &= \Gamma_{X(t)}^{X(a)} \lim_{s \rightarrow t}\frac{ \Gamma_{X(s)}^{X(t)} A(s)- A(t)}{s-t}= \Gamma_{X(t)}^{X(a)} \nabla A(t).
\end{align*}
We can subtract $\Gamma_{X(t)}^{X(a)}$ from the limit, because it is independent of $s$. We can write $\Gamma_{X(a)}^{X(t)} \Gamma_{X(s)}^{X(a)}=\Gamma_{X(s)}^{X(t)}$, because all the parallel transports are along the same curve $X$. Putting all together, we get the result.
\end{proof}

%%%%%%%%%%%%%%%%%%%%%%%%%%%%%%
%%%%%%%%%%%%%%%%%%%%%%%%%%%%%%%%%%%

\section{Proofs of convergence}
\subsection{The convex case}
\convex*
\begin{proof}
Consider the Lyapunov function
\begin{equation*}
    \epsilon(t)=t^2(f(X)-f(x^*))+2\| -\textnormal{\textnormal{log}}_X(x^*)+\frac{t}{2} \dot X \|^2 + 2(\zeta-1) \| \textnormal{\textnormal{log}}_X(x^*) \|^2.
\end{equation*}
We have that 
\begin{align*}
  \dot \epsilon(t)&=2t(f(X)-f(x^*))+t^2\langle \textnormal{\textnormal{gradf}}(X), \dot X \rangle \\ &+ 4\langle -\textnormal{\textnormal{log}}_X(x^*)+\frac{t}{2} \dot X, -\nabla \textnormal{\textnormal{log}}_X(x^*)+\frac{1}{2} \dot X+\frac{t}{2} \nabla \dot X\rangle +4(\zeta-1)\langle -\dot X,\textnormal{\textnormal{log}}_X(x^*) \rangle \\ &= 2t(f(X)-f(x^*))+t^2\langle \textnormal{\textnormal{gradf}}(X), \dot X \rangle \\ &+  4\langle -\textnormal{\textnormal{log}}_X(x^*)+\frac{t}{2} \dot X, -\nabla \textnormal{\textnormal{log}}_X(x^*)-\zeta \dot X+\zeta\dot X+\frac{1}{2} \dot X+\frac{t}{2} \nabla \dot X \rangle +4(\zeta-1)\langle -\dot X,\textnormal{\textnormal{log}}_X(x^*) \rangle \\ &= 2t(f(X)-f(x^*))+t^2\langle \textnormal{\textnormal{gradf}}(X), \dot X \rangle \\ &+  4\langle -\textnormal{\textnormal{log}}_X(x^*)+\frac{t}{2} \dot X, -\nabla \textnormal{\textnormal{log}}_X(x^*)-\zeta \dot X-\frac{t}{2} \textnormal{\textnormal{gradf}}(X) \rangle +4(\zeta-1)\langle -\dot X,\textnormal{\textnormal{log}}_X(x^*) \rangle  \\ &= 2t(f(X)-f(x^*))-2t\langle -\textnormal{\textnormal{log}}_X(x^*), \textnormal{\textnormal{gradf}}(X) \rangle +t^2\langle \textnormal{\textnormal{gradf}}(X), \dot X \rangle-t^2 \langle \textnormal{\textnormal{gradf}}(X), \dot X \rangle \\ &+  4\langle -\textnormal{\textnormal{log}}_X(x^*)+\frac{t}{2} \dot X, -\nabla \textnormal{\textnormal{log}}_X(x^*)-\zeta \dot X \rangle +4(\zeta-1)\langle -\dot X,\textnormal{\textnormal{log}}_X(x^*) \rangle \\ & \leq 4\langle -\textnormal{\textnormal{log}}_X(x^*)+\frac{t}{2} \dot X, -\nabla \textnormal{\textnormal{log}}_X(x^*)-\zeta \dot X \rangle +4(\zeta-1)\langle -\dot X,\textnormal{\textnormal{log}}_X(x^*) \rangle, 
\end{align*}
by geodesic convexity. 

The last expression can be written as
\begin{align*}
    &4 \langle \textnormal{\textnormal{log}}_X(x^*), \nabla \textnormal{\textnormal{log}}_X(x^*) \rangle+4 \zeta\langle \textnormal{\textnormal{log}}_X(x^*), \dot X \rangle +4(\zeta-1)\langle -\dot X,\textnormal{\textnormal{log}}_X(x^*) \rangle+ 2t  (\langle \nabla \textnormal{\textnormal{log}}_X(x^*), -\dot X \rangle- \zeta \| \dot X \|^2) \\ &= 2 \frac{d}{dt}d(X,x^*)^2-2 \zeta \frac{d}{dt} d(X,x^*)^2 +2(\zeta-1) \frac{d}{dt} d(X,x^*)^2+ 2t (\langle \nabla \textnormal{\textnormal{log}}_X(x^*), -\dot X \rangle- \zeta \| \dot X \|^2) \\ &= 2t (\langle \nabla \textnormal{\textnormal{log}}_X(x^*), -\dot X \rangle- \zeta \| \dot X \|^2)\leq 0
\end{align*}
by Lemma \ref{le:result for covariant derivative}.
Thus 
\begin{align*}
t^2(f(X)-f(x^*))\leq \epsilon(t)\leq \epsilon(0)=2\zeta \| \textnormal{\textnormal{log}}_{x_0}(x^*) \|^2,
\end{align*}
and the result follows.
\end{proof}

\subsection{The weakly-quasi-convex case}
\weaklyquasiconvex*
\begin{proof}
Consider the Lyapunov function
\begin{equation*}
    \epsilon(t)=\alpha^2 t^2(f(X)-f(x^*))+2\| -\textnormal{\textnormal{log}}_X(x^*)+\frac{\alpha t}{2} \dot X \|^2 + 2(\zeta-1) \| \textnormal{\textnormal{log}}_X(x^*) \|^2.
\end{equation*}
We have that 
\begin{align*}
  \dot \epsilon(t)&=2\alpha^2t(f(X)-f(x^*))+\alpha^2 t^2\langle \textnormal{\textnormal{gradf}}(X), \dot X \rangle \\ &+ 4\langle -\textnormal{\textnormal{log}}_X(x^*)+\frac{\alpha t}{2} \dot X, -\nabla \textnormal{\textnormal{log}}_X(x^*)+\frac{\alpha}{2} \dot X+\frac{\alpha t}{2} \nabla \dot X\rangle +4(\zeta-1)\langle -\dot X,\textnormal{\textnormal{log}}_X(x^*) \rangle \\ &= 2\alpha^2t(f(X)-f(x^*))+\alpha^2 t^2\langle \textnormal{\textnormal{gradf}}(X), \dot X \rangle \\ &+  4\langle -\textnormal{\textnormal{log}}_X(x^*)+\frac{\alpha t}{2} \dot X, -\nabla \textnormal{\textnormal{log}}_X(x^*)-\zeta \dot X+\zeta\dot X+\frac{\alpha}{2} \dot X+\frac{\alpha t}{2} \nabla \dot X \rangle +4(\zeta-1)\langle -\dot X,\textnormal{\textnormal{log}}_X(x^*) \rangle \\ &= 2 \alpha^2 t(f(X)-f(x^*))+ \alpha^2 t^2\langle \textnormal{\textnormal{gradf}}(X), \dot X \rangle \\ &+  4\langle -\textnormal{\textnormal{log}}_X(x^*)+\frac{\alpha t}{2} \dot X, -\nabla \textnormal{\textnormal{log}}_X(x^*)-\zeta \dot X-\frac{\alpha t}{2} \textnormal{\textnormal{gradf}}(X) \rangle +4(\zeta-1)\langle -\dot X,\textnormal{\textnormal{log}}_X(x^*) \rangle  \\ &= 2\alpha^2 t(f(X)-f(x^*))-2 \alpha t \langle -\textnormal{\textnormal{log}}_X(x^*), \textnormal{\textnormal{gradf}}(X) \rangle +\alpha^2 t^2\langle \textnormal{\textnormal{gradf}}(X), \dot X \rangle-\alpha^2 t^2 \langle \textnormal{\textnormal{gradf}}(X), \dot X \rangle \\ &+  4\langle -\textnormal{\textnormal{log}}_X(x^*)+\frac{\alpha t}{2} \dot X, -\nabla \textnormal{\textnormal{log}}_X(x^*)-\zeta \dot X \rangle +4(\zeta-1)\langle -\dot X,\textnormal{\textnormal{log}}_X(x^*) \rangle \\ & \leq 4\langle -\textnormal{\textnormal{log}}_X(x^*)+\frac{\alpha t}{2} \dot X, -\nabla \textnormal{\textnormal{log}}_X(x^*)-\zeta \dot X \rangle +4(\zeta-1)\langle -\dot X,\textnormal{\textnormal{log}}_X(x^*) \rangle  
\end{align*}
by geodesic $\alpha$-weak-quasi-convexity.
The last expression can be written as
\begin{align*}
    &4 \langle \textnormal{\textnormal{log}}_X(x^*), \nabla \textnormal{\textnormal{log}}_X(x^*) \rangle+4 \zeta\langle \textnormal{\textnormal{log}}_X(x^*), \dot X \rangle +4(\zeta-1)\langle -\dot X,\textnormal{\textnormal{log}}_X(x^*) \rangle+ 2 \alpha t (\langle \nabla \textnormal{\textnormal{log}}_X(x^*), -\dot X \rangle- \zeta \| \dot X \|^2) \\ &= 2 \frac{d}{dt}d(X,x^*)^2-2 \zeta \frac{d}{dt} d(X,x^*)^2 + 2 (\zeta-1) \frac{d}{dt} d(X,x^*)^2+ 2 \alpha t (\langle \nabla \textnormal{\textnormal{log}}_X(x^*), -\dot X \rangle- \zeta \| \dot X \|^2) \\ &= 2 \alpha t (\langle \nabla \textnormal{\textnormal{log}}_X(x^*), -\dot X \rangle- \zeta \| \dot X \|^2)\leq 0
\end{align*}
by Lemma \ref{le:result for covariant derivative}.
Thus, 
\begin{align*}
\alpha^2 t^2(f(X)-f(x^*))\leq \epsilon(t)\leq \epsilon(0)=2\zeta \| \textnormal{\textnormal{log}}_{x_0}(x^*) \|^2    
\end{align*}
and the result follows.
\end{proof}\newpage

\subsection{The strongly convex case}
\stronglyconvex*
\begin{proof}
Consider the Lyapunov function
\begin{equation*}
    \epsilon(t)= e^{\sqrt{\frac{\mu}{\zeta}}t}\left(\frac{\mu}{2 \zeta} \| -\textnormal{\textnormal{log}}_X(x^*)+\sqrt{\frac{\zeta}{\mu}} \dot X \|^2+f(X)-f^*+\frac{\mu (\zeta-1)}{2 \zeta} \| \textnormal{\textnormal{log}}_X(x^*)\|^2\right).
\end{equation*}
We have that 
\begin{align*}
    &\frac{d}{dt} \left(e^{\sqrt{\frac{\mu}{\zeta}}t}\left(\frac{\mu}{2 \zeta} \| -\textnormal{\textnormal{log}}_X(x^*)+\sqrt{\frac{\zeta}{\mu}} \dot X \|^2+\frac{\mu (\zeta-1)}{2 \zeta}\| \textnormal{\textnormal{log}}_X(x^*) \|^2\right)\right) \\ &=
    \sqrt{\frac{\mu}{\zeta}}\frac{\mu}{\zeta}e^{\sqrt{\frac{\mu}{\zeta}}t}\left(\frac{1}{2} \| -\textnormal{\textnormal{log}}_X(x^*)+\sqrt{\frac{\zeta}{\mu}} \dot X \|^2+ \frac{\zeta-1}{2} \| \textnormal{\textnormal{log}}_X(x^*) \|^2\right) \\ &+ 
    \frac{\mu}{\zeta}e^{\sqrt{\frac{\mu}{\zeta}}t} \left(\langle -\textnormal{\textnormal{log}}_X(x^*)+\sqrt{\frac{\zeta}{\mu}} \dot X, -\nabla \textnormal{\textnormal{log}}_X(x^*)+
    \sqrt{\frac{\zeta}{\mu}} \nabla \dot X \rangle+(\zeta-1) \langle \textnormal{\textnormal{log}}_X(x^*), -\dot X \rangle\right) \\ &=
    \sqrt{\frac{\mu}{\zeta}}\frac{\mu}{\zeta}e^{\sqrt{\frac{\mu}{\zeta}}t}\left(\frac{1}{2} \| -\textnormal{\textnormal{log}}_X(x^*)+\sqrt{\frac{\zeta}{\mu}} \dot X \|^2+ \frac{\zeta-1}{2} \| \textnormal{\textnormal{log}}_X(x^*) \|^2\right) \\ &+
    \frac{\mu}{\zeta}e^{\sqrt{\frac{\mu}{\zeta}}t} \left(\langle -\textnormal{\textnormal{log}}_X(x^*)+\sqrt{\frac{\zeta}{\mu}} \dot X, -\nabla \textnormal{\textnormal{log}}_X(x^*)-\zeta \dot X + \zeta \dot X+
    \sqrt{\frac{\zeta}{\mu}} \nabla \dot X \rangle +(\zeta-1) \langle \textnormal{\textnormal{log}}_X(x^*), -\dot X \rangle\right) \\ &= 
    \sqrt{\frac{\mu}{\zeta}}\frac{\mu}{\zeta}e^{\sqrt{\frac{\mu}{\zeta}}t}\left(\frac{1}{2} \| -\textnormal{\textnormal{log}}_X(x^*)+\sqrt{\frac{\zeta}{\mu}} \dot X \|^2+ \frac{\zeta-1}{2} \| \textnormal{\textnormal{log}}_X(x^*) \|^2\right) \\ &+
    \frac{\mu}{\zeta}e^{\sqrt{\frac{\mu}{\zeta}}t} \left(\langle -\textnormal{\textnormal{log}}_X(x^*)+\sqrt{\frac{\zeta}{\mu}} \dot X, -\nabla \textnormal{\textnormal{log}}_X(x^*)-\zeta \dot X \rangle +(\zeta-1) \langle \textnormal{\textnormal{log}}_X(x^*), -\dot X \rangle\right.\\ &+ \left.
    \langle -\textnormal{\textnormal{log}}_X(x^*)+\sqrt{\frac{\zeta}{\mu}} \dot X, -\dot X-\sqrt{\frac{\zeta}{\mu}} \textnormal{\textnormal{gradf}}(X) \rangle\right).
\end{align*}
The expression
\begin{equation*}
    \frac{\mu}{\zeta}e^{\sqrt{\frac{\mu}{\zeta}}t} \left(\langle -\textnormal{\textnormal{log}}_X(x^*)+\sqrt{\frac{\zeta}{\mu}} \dot X, -\nabla \textnormal{\textnormal{log}}_X(x^*)-\zeta \dot X \rangle +(\zeta-1) \langle \textnormal{\textnormal{log}}_X(x^*), -\dot X \rangle\right)
\end{equation*}
is equal to 
\begin{align*}
    \frac{\mu}{\zeta}e^{\sqrt{\frac{\mu}{\zeta}}t} ((1-\zeta) (d(X,x^*)^2)'+(\zeta-1)(d(X,x^*)^2)'+\sqrt{\frac{\zeta}{\mu}}(\langle -\nabla \textnormal{\textnormal{log}}_X(x^*), \dot X \rangle-\zeta \| \dot X \|^2)) \leq 0
\end{align*}
by Lemma \ref{le:result for covariant derivative}.
\newline
Thus, we have
\begin{align*}
    &\frac{d}{dt} \left(e^{\sqrt{\frac{\mu}{\zeta}}t}\left(\frac{\mu}{2 \zeta} \| -\textnormal{\textnormal{log}}_X(x^*)+\sqrt{\frac{\zeta}{\mu}} \dot X \|^2+\frac{\mu (\zeta-1)}{2 \zeta} \| \textnormal{\textnormal{log}}_X(x^*) \|^2\right)\right) \\ & \leq
    \sqrt{\frac{\mu}{\zeta}}\frac{\mu}{\zeta}e^{\sqrt{\frac{\mu}{\zeta}}t}\left(\frac{1}{2} \| -\textnormal{\textnormal{log}}_X(x^*)+\sqrt{\frac{\zeta}{\mu}} \dot X \|^2+\langle -\textnormal{\textnormal{log}}_X(x^*)+\sqrt{\frac{\zeta}{\mu}} \dot X , -\sqrt{\frac{\zeta}{\mu}}\dot X\rangle\right.  \\ &\left.+\frac{\zeta-1}{2} \| \textnormal{\textnormal{log}}_X(x^*) \|^2\right)+ \frac{\mu}{\zeta}e^{\sqrt{\frac{\mu}{\zeta}}t} \langle -\textnormal{\textnormal{log}}_X(x^*)+\sqrt{\frac{\zeta}{\mu}} \dot X, -\dot X-\sqrt{\frac{\zeta}{\mu}} \textnormal{\textnormal{gradf}}(X) \rangle \\ &=
    \sqrt{\frac{\mu}{\zeta}}\frac{\mu}{\zeta}e^{\sqrt{\frac{\mu}{\zeta}}t}
    \left(-\frac{1}{2} \| \sqrt{\frac{\mu}{\zeta}} \dot X \|^2+\frac{1}{2} \| \textnormal{\textnormal{log}}_X(x^*) \|^2\right. \\ &\left.+ \frac{\zeta-1}{2} \| \textnormal{\textnormal{log}}_X(x^*) \|^2\right) + \frac{\mu}{\zeta}e^{\sqrt{\frac{\mu}{\zeta}}t} \langle -\textnormal{\textnormal{log}}_X(x^*)+\sqrt{\frac{\zeta}{\mu}} \dot X, -\dot X-\sqrt{\frac{\zeta}{\mu}} \textnormal{\textnormal{gradf}}(X) \rangle \\ & \leq
    \sqrt{\frac{\mu}{\zeta}}e^{\sqrt{\frac{\mu}{\zeta}}t} \frac{\mu}{2} \| \textnormal{\textnormal{log}}_X(x^*) \|^2+\sqrt{\frac{\mu}{\zeta}}e^{\sqrt{\frac{\mu}{\zeta}}t}\langle \textnormal{\textnormal{log}}_X(x^*), \textnormal{\textnormal{gradf}}(X) \rangle-e^{\sqrt{\frac{\mu}{\zeta}}t} \langle \textnormal{\textnormal{gradf}}(X), \dot X \rangle \\ &\leq
    -\sqrt{\frac{\mu}{\zeta}}e^{\sqrt{\frac{\mu}{\zeta}}t}(f(X)-f^*)-e^{\sqrt{\frac{\mu}{\zeta}}t} \langle \textnormal{\textnormal{gradf}}(X), \dot X \rangle \\ &=
    \frac{d}{dt} \left(-e^{\sqrt{\frac{\mu}{\zeta}}t} (f(X)-f^*)\right),
\end{align*}
where the last inequality follows from geodesic $\mu$-strong convexity of $f$. Thus, $\dot \epsilon (t) \leq 0$ and the result follows.
\end{proof}

%%%%%%%%%%%%%%%%%%%%%%%%%%
%%%%%%%%%%%%%%%%%%%%%%%%%%%

\section{Proofs about shadowing}
\subsection{Pseudo-orbit property}

In this subsection, we prove that the continuous-time limit of Riemannian gradient descent ($y=-\textnormal{gradf}(y)$) returns a pseudo-orbit of Riemannian gradient descent. This result is standard in numerical analysis~(error of Euler integration), and can be also found (in a less general form) as Proposition 2 in~\cite{absil2012projection}.  

We recall that, in analogy with~\cite{Antonioshadowing}, we assume that $f:M \rightarrow R$ is a $C^2$ function such that for all points on the ODE solution, $\| \textnormal{gradf}(x) \| \leq \ell$ and $\mu \leq \| \textnormal{Hessf}(x) \| \leq L$.
\pseudoorbit*
\begin{proof}
We consider the curve $y:[kh,kh+h] \rightarrow M$ which is the solution of the gradient flow ODE $\dot y=-\textnormal{gradf}(y)$ and the geodesic $\gamma(t-kh)=\exp_{y(kh)}(t \dot y(kh))$, which has the same initial velocity with $y$. Here $y_k=y(kh)$ and $y_{k+1}=y((k+1)h)$.
\newline
The manifold $M$ can be considered as a submanifold of $\mathbb{R}^n$ for sufficiently large $n$, up to isometry, because of the Nash-Kuiper embedding theorem. Thus we can expand $y$ and $\gamma$ using a Taylor series in the ambient space:
\begin{align*}
    & y(kh+h)=y(kh)+h \dot y(kh)+\int_{kh}^{kh+h} \frac{\ddot y(t)}{2}(kh+h-t)^2 dt; \\ &
    \gamma(kh+h)=\gamma(kh)+h \dot \gamma(kh)+\int_{kh}^{kh+h} \frac{\ddot \gamma(t)}{2}(kh+h-t)^2 dt.
\end{align*}
We have that $y(kh)=\gamma(kh)$, $\dot y(kh)= \dot \gamma(kh)$, thus
\begin{equation*}
    y(kh+h)-\gamma(kh+h)=\int_{kh}^{kh+h} \frac{\ddot y(t)-\ddot \gamma(t)}{2}(kh+h-t)^2 dt
\end{equation*}
and
\begin{equation*}
    \| y(kh+h)-\gamma(kh+h) \| \leq \frac{h^2}{2} \| \ddot y(t_0)-\ddot \gamma(t_0) \|. 
\end{equation*}
by the mean value theorem for integrals, where $t_0 \in (kh, kh+h)$.
\newline
One can easily check that
\begin{equation*}
    d(y(kh+h),\gamma(kh+h)) \leq A \| y(kh+h)-\gamma(kh+h) \|,
\end{equation*}
where $A$ is a constant depending on the bound $D$ of the working domain.
\newline
Indeed let two points $p,q \in M$. If $p$ and $q$ are very close together then the ratio of the intrinsic to extrinsic distance tends to $1$, check lemma 4.2.7 in \cite{salamon2019}. Obviously $d(p,q) \rightarrow 0$ if and only if $\| p-q \| \rightarrow 0$. Thus if $d(p,q) \geq a >0$, then $\| p-q \| \geq b>0$. But our working domain is bounded and $d(p,q) \leq D$. Thus $d(p,q) \leq \frac{D}{b} \| p-q \|$. 
\newline
The term $\| \ddot y(t_0)-\ddot \gamma(t_0) \|$ can be written by the Gauss-Weingarden formula as
\begin{equation*}
    \| \nabla \dot y(t_0)+\mathcal{L}_{y(t_0)}(\dot y(t_0), \dot y(t_0))-\mathcal{L}_{\gamma(t_0)}(\dot \gamma(t_0), \dot \gamma(t_0)) \|
\end{equation*}
where $\mathcal{L}$ is the second fundamental form.
The last expression is less or equal than
\begin{equation*}
    \| \nabla \dot y(t_0) \|+\| \mathcal{L}_{y(t_0)}(\dot y(t_0), \dot y(t_0)) \|+ \| \mathcal{L}_{\gamma(t_0)}(\dot \gamma(t_0), \dot \gamma(t_0)) \|.
\end{equation*}
We have that
\begin{equation*}
    \| \nabla \dot y(t_0) \| \leq \| \textnormal{Hessf}(y(t_0)) \| \| \textnormal{gradf}(y(t_0)) \| \leq \ell L
\end{equation*}
by our initial assumptions ($\ell$ bounds the Riemannian gradient and $f$ is $L$-smooth).
\newline
Finally, it is known that the second fundamental form $\mathcal{L}$ is bounded if the sectional curvatures are bounded from above and below (in our case are just constant equal to $K$) and the injectivity radius is bounded from below (this is the case for us, since we assume that there exist always some geodesic connecting any two points in $M$). For a discussion of this fact check \cite{petrunin}.
\newline
Thus
\begin{equation*}
    d(y(kh+h),\exp_{y(kh)}(h \dot y(kh)) \leq C h^2
\end{equation*}
where $C$ is constant depending only on the curvature $K$, the dimension of the manifold $M$, $\ell$ and $L$.
\end{proof}

\subsection{Contraction of RGD}
We start with important computations for Jacobi fields in symmetric manifolds. The reader can refer to \cite{leimkuhler1996symplectic} (Section 4.3) and \cite{Klingenberg_textbook} (Section 2.2). According to these references, the Jacobi field $J$ of a symmetric manifold along the geodesic $\textnormal{exp}_p(tw)$ with initial conditions $J(0)=a$ and $\nabla J(0)=b$ is given by the formula:
\begin{equation*}
    J(t)= \Gamma_p^{tw}\left(f_1(t^2 R_w)a+f_2(t^2 R_w) b\right)
\end{equation*}
where
\begin{equation*}
    f_1(z)=\cos(\sqrt{z}),
\end{equation*}
\begin{equation*}
    f_2(z)=\frac{\sin(\sqrt{z})}{\sqrt{z}},
\end{equation*}
and
$R_w:T_p M \rightarrow T_p M$ is defined by
\begin{equation*}
    R_w(u)=R(u,w)w,
\end{equation*}
where $R$ is the Riemann curvature tensor.
\begin{lemma}
A Riemannian manifold $M$ has constant curvature $K$ if and only if
\begin{equation*}
    \langle R(u_1,u_2)u_3,u_4 \rangle=K(\langle u_1,u_4 \rangle \langle u_2,u_3 \rangle - \langle u_1,u_3 \rangle \langle u_2,u_4 \rangle).
\end{equation*}
\end{lemma}
Thus 
\begin{equation*}
\begin{split}
    \langle R_w(u),v \rangle= 
    \langle R(u,w)w,v \rangle &= K(\langle u,v \rangle \| w \|^2- \langle u,w \rangle \langle w,v \rangle) = 
    \langle K (\| w \| ^2 u -\langle u,w \rangle w), v \rangle
    \end{split}
\end{equation*}
and we derive that $R_w(u)= K (\| w \|^2 u - \langle u,w \rangle w)$. Define 
$A(u)=\| w \|^2 u - \langle u,w \rangle w$.
Now we must compute the powers of $A$: 
\begin{equation*}
\begin{split}
    A^2(u) & =A(A(u))=\| w \|^2 A(u) - \langle A(u),w \rangle w \\ &=\| w \|^2 u - \langle u,w \rangle w) - \langle \| w \|^2 u - \langle u,w \rangle w,w \rangle w \\ &= \| w \|^2 (\| w \|^2 u - \langle u,w \rangle w)-\| w \| ^2 \langle u,w \rangle w- \langle u,w \rangle \| w \| ^2 w=  \| w \| ^2 A(u).
    \end{split}
\end{equation*}
By induction we conclude that $A^l(u)=\| w \|^{2l-2} A(u)$ for $l \geq 1$ and $A^0(u)= \textnormal{Id}$. This implies that $R_w^l(u)=(K \| w \| ^2)^l\frac{1}{\| w \| ^2}A(u)$, for $l \geq 1$ and $R_w^0(u)=\textnormal{Id}$.
Now we feed $R_w$ to the operators $f_1(z)$ and $f_2(z)$. We extend both these two expressions in power series and we get
\begin{equation*}
\begin{split}
   f_1(R_w) & =\textnormal{Id}+\sum_{l=1}^{\infty} \frac{(-1)^l}{(2l)!} R_w^l =\textnormal{Id}+\sum_{l=1}^{\infty} \left(\frac{(-1)^l}{(2l)!} (K \| w \| ^2)^l\frac{1}{\| w \| ^2} A\right) \\ &= \textnormal{Id}+ \frac{1}{\| w \| ^2}A\left( \sum_{l=0}^{\infty} \left(\frac{(-1)^l}{(2l)!} (K \| w \| ^2)^l\right)-1\right)=\textnormal{Id}+ \frac{1}{\| w \|}A\left( \cos(\sqrt{K}\| w \|\right)-1. 
   \end{split}
\end{equation*}
In exactly the same way we get that
\begin{equation*}
    f_2(R_w)=\textnormal{Id}+ \frac{1}{\| w \| ^2}A\left( \frac{\sin(\sqrt{K}\| w \|)}{\sqrt{K}\| w \|)}-1\right).
\end{equation*}
In our case $w$ is the vector defining a geodesic.
\newline We compute now the operator $f_2(R_w)^{-1}=\left(\textnormal{Id}+\frac{1}{\| w \|^2} \left(\frac{\sin(a)}{a}-1\right) A\right)^{-1}$, where $a=\sqrt{K}\| w \|$.
We have that $f_2(R_w)^{-1}(u)=\frac{a}{\sin(a)} u+\frac{1}{\| w \|^2} \left(1-\frac{a}{\sin(a)}\right)\langle u,w \rangle w$.
\begin{align*}
    \text{Check:}  \ \ \ \ &f_2(R_w)^{-1}(f_2(R_w)(u))   \\ &= \frac{a}{\sin(a)}\left(\frac{\sin(a)}{a}u-\frac{1}{\| w \|^2}\left(\frac{\sin(a)}{a}-1\right)\langle u,w \rangle w\right)\\& \ \ \ + \frac{1}{\| w \|^2}\left(1-\frac{a}{\sin(a)}\right) \langle \frac{\sin(a)}{a} u-\frac{1}{\| w \| ^2}\left(\frac{\sin(a)}{a}-1\right)\langle u,w \rangle w,w\rangle \\&=
    u-\frac{1}{\| w \|^2}\left(1-\frac{a}{\sin(a)}\right) \langle u,w \rangle w\\
    & \ \ \ +\frac{1}{\| w \|^2}\left(\frac{\sin(a)}{a}-1\right) \langle u,w \rangle w - \frac{1}{\| w \|^2} \left(1-\frac{a}{\sin(a)}\right)\left(\frac{\sin(a)}{a}-1\right)\langle u,w \rangle \| w \|^2 w \\ &= 
    u-\frac{1}{\| w \|^2}(1-\frac{a}{\sin(a)}) \langle u,w \rangle w+\frac{1}{\| w \|^2}(1-\frac{a}{\sin(a)}) \langle u,w \rangle w\\&=u,
\end{align*}
because $f_2(R_w)(u)=\frac{\sin(a)}{a}u-\frac{1}{\| w \|^2}\left(\frac{\sin(a)}{a}-1\right) \langle u,w \rangle w$.
We are interested in the norm of the operator $f_2(R_w)$. It is easy to show that it is a self-adjoint operator, thus remains to compute its eigenvalues.
We solve the equation
\begin{equation*}
    f_2(R_w)^{-1}(u)=bu
\end{equation*}
for real numbers $b$. It becomes 
\begin{equation*}
    f_2(R_w)^{-1}(u)=\frac{a}{\sin(a)} u+\frac{1}{\| w \|^2} \left(1-\frac{a}{\sin(a)}\right)\langle u,w \rangle w=bu.
\end{equation*}
    If $u \in w^{\bot}$, then $b=\frac{a}{\sin(a)}$. Since $\textnormal{dim}(w^{\bot})=\textnormal{dim}(M)-1$, $\frac{a}{\sin(a)}$ is an eigenvalue of $f_2(R_w)^{-1}$ of multiplicity $\textnormal{dim}(M)-1$. If $u \in <w>$, then $b=1$. Since $\textnormal{dim}(<w>)=1$, $1$ is an eigenvalue of $f_2(R_w)^{-1}$ with multiplicity $1$, and there are no other eigenvalues.
\newline
In addition, the operators $f_1(R_w)$ and $f_2(R_w)$ are self-adjoint. We briefly check it for the first one. We have
\begin{align*}
    f_1(R_w)(u)=\cos(a) u+\frac{1}{\| w \|^2}(\cos(a)-1) (\| w \|^2 u-\langle u,w \rangle w)
\end{align*}
and
\begin{multline*}
    \langle f_1(R_w)u,v \rangle= \cos(a) \langle u,v \rangle +\frac{1}{\| w \|^2} (\cos(a)-1) \langle u,w \rangle \langle w,v \rangle\\ = \cos(a) \langle v,u \rangle +\frac{1}{\| w \|^2} (\cos(a)-1) \langle v,w \rangle \langle w,u \rangle=\langle u, f_1(R_w)v \rangle.
\end{multline*}

\contraction*
\begin{proof}
Denote $a_1=\textnormal{exp}_{x_1}(-h \textnormal{gradf}(x_1)), a_2= \textnormal{exp}_{x_2}(-h \textnormal{gradf}(x_2))$. We denote the geodesic connecting $x_1$ and $x_2$ by $X$ and create a variation of geodesics defined by $\textnormal{exp}_{X(t)}(uE(t))$ where $E(t)$ is a vector field along $X$ with $E(0)=-h \textnormal{gradf}(x_1)$ and $E(1)=-h \textnormal{gradf}(x_2)$. We have that $d(a_1,a_2)$ is equal to the length of the geodesic connecting $a_1$ and $a_2$, which is less or equal than the length of the curve $\beta(t)=\textnormal{exp}_{X(t)}(E(t))$, because $\beta(0)=a_1$ and $\beta(1)=a_2$. Thus
\begin{align*}
    d(a_1,a_2) \leq \int_0^1 \| \dot \beta(t) \| dt= \| \dot \beta(t_0) \|,
\end{align*}
for some $t_0 \in (0,1)$.
\newline
By the construction of Jacobi fields as measures of variations through geodesics, we have that $\dot \beta(t_0)$ is equal to $J(1)$ where $J$ is the Jacobi field with initial conditions $J(0)=\dot X(t_0)$ and $\nabla J(0)= \nabla E(t_0)$. A valid choice for $E(t)$ is $-h \textnormal{gradf}(X(t))$, thus $\nabla E(t_0)=-h \textnormal{Hess}f(X(t_0)) \dot X(t_0)$. 
\newline
In a complete, connected, simply connected manifold of constant curvature $K$ (i.e. symmetric) we can compute the Jacobi field $J$ precisely. Namely if $J$ is the Jacobi field along the geodesic from $x(t_0)$ to $\beta(t_0)$ with initial conditions $J(0)=a$ and $J(1)=b$, we have
\begin{equation*}
    J(t)= \Gamma_{X(t_0)}^{t w} (f_1(t^2 R_w)a+f_2(t^2R_w) b )
\end{equation*}
where $f_1(z)=\cos(\sqrt{z})$ and $f_2(z)=\frac{\sin(\sqrt{z})}{\sqrt{z}}$ and $w=\textnormal{log}_{X(t_0)}(\beta(t_0))$.
\newline
By our computations for $R_w,f_1$ and $f_2$ above, we get
\begin{align*}
    \| J(1) \|= \| f_1(R_w) \dot X(t_0)- h f_2(R_w) \textnormal{Hessf}(X(t_0)) \dot X(t_0) \| \leq \| f_2(R_w) \| \| \frac{f_1(R_w)}{f_2(R_w)}-h \textnormal{Hessf}(X(t_0)) \| \| \dot X(t_0) \|.
\end{align*}
The eigenvalues of $f_1(R_w)$ are $1$ and $\cos(d)$, while of $\frac{f_1(R_w)}{f_2(R_w)}$, $1$ and $d \cot(d)$, where $d=\sqrt{K} \| w \|$.
The operator $-h \textnormal{Hess}f(X(t_0))$ is symmetric (because it is taken with respect to the Levi-Civita connection, which is torsion-free) and its largest eigenvalue is $-h \mu$. People have proved that the largest eigenvalue of the sum of two hermitian operators is at most the sum of the two largest eigenvalues respectively (check for instance \cite{fulton1998equivalent}). Thus we consider cases regarding the curvature $K$.
\begin{itemize}
    \item If $K \geq 0$, then the largest eigenvalues of $f_2(R_w)$ and $\frac{f_1(R_w)}{f_2(R_w)}$ are both 1. Thus 
    \begin{equation*}
        \| J(1) \| \leq (1-h \mu) \| \dot X(t_0) \|.
    \end{equation*}
    \item If $K<0$, then the largest eigenvalue of $f_2(R_w)$ is $\frac{\sin(d)}{d}$ and of $\frac{f_1(R_w)}{f_2(R_w)}$ is $d \cot(d)$. Thus
    \begin{equation*}
      \| J(1) \| \leq \frac{\sin(d)}{d} (d \cot(d)-h \mu) \| \dot X(t_0) \| \leq \frac{\sinh(\sqrt{-K}D}{\sqrt{-K}D} (\sqrt{-K}D \coth(\sqrt{-K}D)-h \mu) \| \dot X(t_0) \|,  
    \end{equation*}
    where $D$ is an upper bound for the working domain.
\end{itemize}
\begin{figure}
    \centering
    \includegraphics[width=0.5\linewidth]{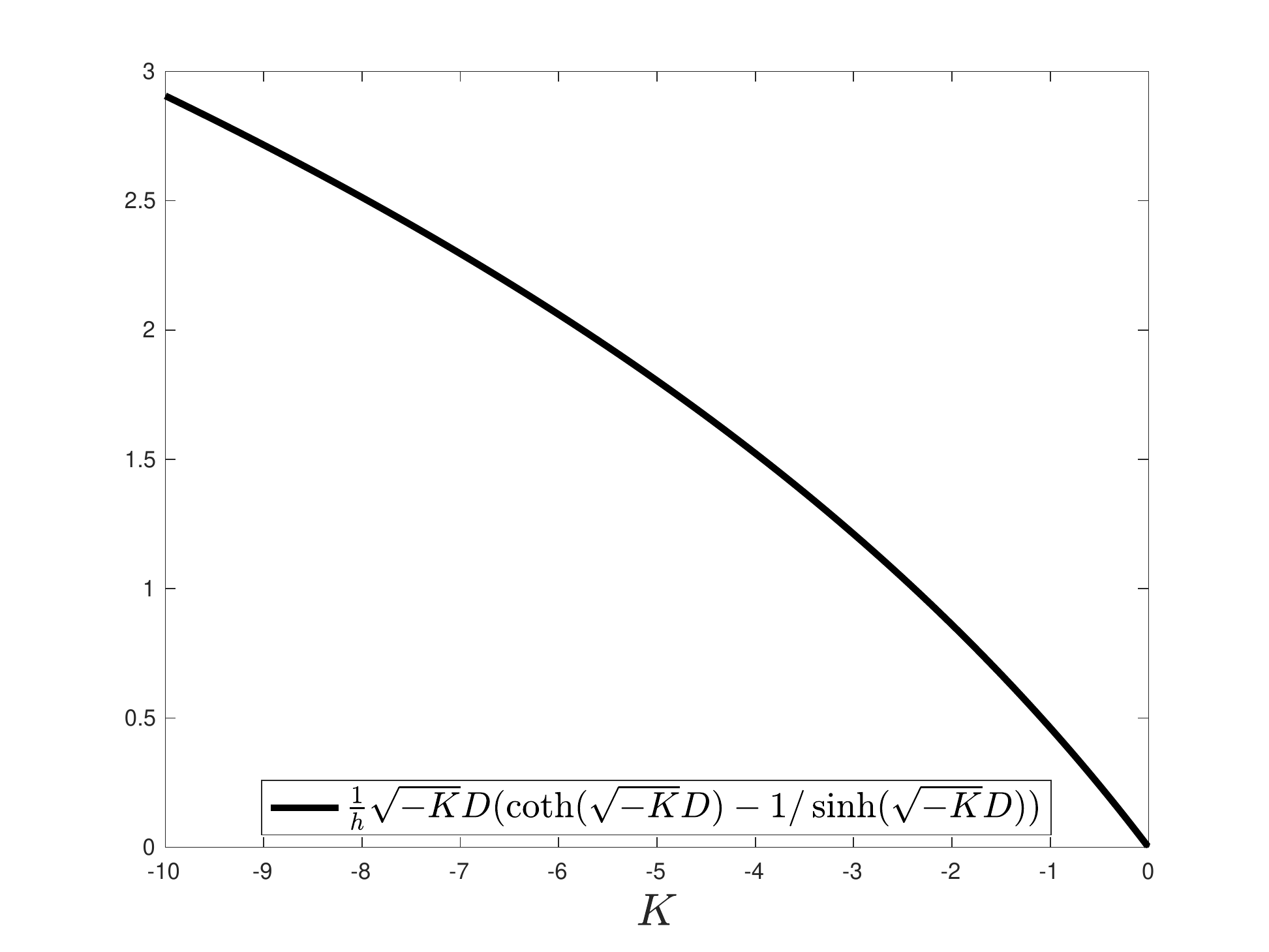}
    \caption{The lower bound for $\mu$ (Equation~\ref{eq:bound_on_mu}) in negative curvature, plotted for $D = h = 1$}
    \end{figure}
Finally $\| \dot X(t_0) \|=\| \dot X(0) \|= d(x_1,x_2)$, because $X$ is the geodesic connecting $x_1$ and $x_2$ (thus it has constant speed).
\end{proof}
\shadowing*
\begin{proof}
If
\begin{equation}
    \mu > \frac{1}{h} \sqrt{-K}D \left(\coth(\sqrt{-K}D)-\frac{1}{\sinh(\sqrt{-K}D)}\right), 
    \label{eq:bound_on_mu}
\end{equation}
then $\xi<1$ and Riemannian gradient descent is contracting. For this to hold we need extra to assume that $K$ and $D$ are chosen, such that $\frac{1}{h} \sqrt{-K}D (\coth(\sqrt{-K}D)-\frac{1}{\sinh(\sqrt{-K}D)})<L$.
Let $\epsilon>0$ be the desired tracking accuracy, to be restricted further later. By the contraction shadowing theorem, an orbit generated by Riemannian gradient flow is $\epsilon$-shadowed by a $\delta$-pseudo-orbit generated by Riemannian gradient descent, such that $\delta \leq (1-\xi) \epsilon$. Since $\delta \leq C h^2$, we need $C h^2 \leq (1-\xi) \epsilon$. Substituting $\xi=\lambda(\zeta-h \mu)$, we get the quadratic inequality:
\begin{equation*}
    h^2-\frac{\lambda \mu \epsilon}{C} h+\frac{(\lambda \zeta -1) \epsilon}{C} \leq 0.
\end{equation*}
This inequality has a solution if
\begin{equation*}
    \epsilon \geq \frac{4 C (\lambda \zeta-1)}{\lambda^2 \mu^2}.
\end{equation*}
Given this condition for $\epsilon$ we have that
\begin{equation*}
 h \leq \left(\frac{\lambda \mu}{2C }+\sqrt{\frac{\lambda^2 \mu^2}{4C^2}-\frac{\lambda \zeta-1}{C \epsilon}}\right) \epsilon   
\end{equation*}
Finally, taking into consideration that $h \leq \frac{1}{L}$ we get the result.
\end{proof}